\documentclass[12pt]{amsart}

\usepackage{lscape,amssymb, amsmath,verbatim,graphicx}
\usepackage{epsfig,subfigure,float}
\usepackage{graphicx}
\usepackage{booktabs}
\usepackage{url}
\usepackage[round]{natbib}

\newcommand{\intt}{\int\hspace{-.2cm}\int}

\newtheorem{theorem}{Theorem}[section]
\newtheorem{lemma}{Lemma}[section]
\newtheorem{corollary}{Corollary}
\newcommand{\norm}{|\hspace{-.5mm}|}
\newcommand{\la}{\langle}
\newcommand{\ra}{\rangle}

\newcommand{\bxi}{{\boldsymbol \xi}}
\newcommand{\bpsi}{{\boldsymbol \psi}}
\newcommand{\bkappa}{{\boldsymbol \kappa}}
\newcommand{\bgamma}{{\boldsymbol \gamma}}

\newtheorem{assumption}{Assumption}[section]
\newtheorem{example}{Example}[section]
\theoremstyle{definition}

\allowdisplaybreaks
\def\beq{\begin{equation}}
\def\eeq{\end{equation}}

\numberwithin{equation}{section}
\numberwithin{theorem}{section}
\numberwithin{table}{section}
\numberwithin{remark}{section}
\numberwithin{corollary}{section}

\renewcommand{\baselinestretch}{1.1}

\begin{document}

\title[Functional data analysis with increasing number of projections]
{Functional data analysis with increasing number of projections}
\author{Stefan Fremdt}
\address{Stefan Fremdt, Mathematical Institute,
University of Cologne, Weyertal 86--90, D--50931 K\"oln, Germany
}
\author {Lajos Horv\'ath}$^1$
\address{Lajos Horv\'ath, Department of Mathematics,
University of Utah, Salt Lake City, UT 84112--0090,  USA
}

\author{Piotr Kokoszka}
\address{Piotr Kokoszka, Department of Statistics,
 Colorado State University, Ft.\ Collins, CO 80523--1877, USA}
\author{Josef G.\ Steinebach}
\address{Josef G.\ Steinebach, Mathematical Institute,
University of Cologne, Weyertal 86--90, D--50931 K\"oln, Germany
}


\thanks{Research supported by NSF grants DMS 0905400, DMS 0931948
 and DFG grant STE 306/22-1. }

\begin{abstract}
Functional principal components (FPC's)  provide the most important and
most extensively used tool for dimension reduction and inference
for functional data. The selection of the number, $d$,  of the FPC's to be used
in a specific procedure has attracted a fair amount of attention,
and a number of reasonably effective approaches exist. Intuitively,
they assume that the functional data can be sufficiently well
approximated by a projection onto a finite--dimensional subspace,
and the error resulting from such an approximation does not impact
the conclusions. This  has been shown to be a very effective approach,
but it is desirable to understand the behavior of many inferential
procedures by considering the projections on subspaces
 spanned by an increasing number
of the FPC's. Such an approach reflects more fully the infinite--dimensional
nature of functional data, and allows to derive procedures which are
fairly insensitive to the selection of $d$.
This is accomplished by considering limits as $d\to\infty$ with the
sample size.

We propose a specific framework in which we let $d\to\infty$ by
deriving a normal approximation for the partial sum process
\[
\sum_{j=1}^{\lfloor du\rfloor} \sum_{i=1} ^{\lfloor Nx\rfloor}
\xi_{i,j}, \ \ \ 0\le u\le 1, \ \ 0\le x\le 1,
\]
where $N$ is the sample size and $\xi_{i,j}$ is the score of the $i$th
function with respect to the $j$th FPC. Our approximation can be
used to derive statistics that use segments of observations and
segments of the FPC's. We apply our general results to derive
two inferential procedures for the mean function: a change--point test
and a two--sample test. In addition to the asymptotic theory, the
tests are assessed through a small simulation study and a data example.


\end{abstract}

\maketitle


\newpage

\section{\bf  Introduction}\label{s:intro}
\setcounter{equation}{0} Functional data analysis has grown into a
comprehensive and useful field of statistics which provides a
convenient framework to handle some high--dimensional data structures,
including curves and images.  The monograph of
\citet{ramsay:silverman:2005} has done a lot to introduce its ideas to
the statistics community and beyond.  Several other monographs and
thousands of papers followed.  This paper focuses on a specific aspect
of the mathematical foundations of functional data analysis, which is
however of fairly central importance.  We first describe the contribution of this
paper in broad terms, and provide some more detailed background and
discussion in the latter part of this section.

Perhaps the most important, and definitely the most commonly used,
tool for dimension reduction of functional data is the principal
component analysis. Suppose we observe
a sample of functions, $X_1, X_2, \ldots, X_N$, and denote
by
\[
\hat{\eta}_{i,j}=\int \left ( X_i(t) - \bar{X}_N(t) \right ) \hat{v}_j(t)dt,
\ \ \ \ i=1, 2, \ldots, N,
\ \  j=1,2, \ldots, d,
\]
the scores of the $X_i$ with respect to the estimated
functional principal components $\hat{v}_j$. The scores
$\hat{\eta}_{i,j}$ depend on two variables $i$ and $j$, and
to reflect the infinite--dimensional nature of the data, it may be desirable
to consider asymptotics in which both $N$ and $d$ increase.
This paper establishes results that allow us
to study the two--dimensional partial sum process
\[
\sum_{j=1}^{\lfloor du\rfloor} \sum_{i=1} ^{\lfloor Nx\rfloor}
\int \left ( X_i(t) - \mu_X(t) \right ) {v}_j(t)dt,
\ \ \ 0\le u\le 1, \ \ 0\le x\le 1.
\]
More specifically, we derive a uniform normal
approximation and apply it to two problems related to testing the null
hypothesis that all observed curves have the same mean function. We
obtain new test statistics in which the number of the functional
principal components, $d$, increases slowly with the sample size $N$.
We hope that our general approach will be used to derive similar results
in other settings.

\begin{figure}
 \centering
 \includegraphics[scale = 0.775,angle = 270]{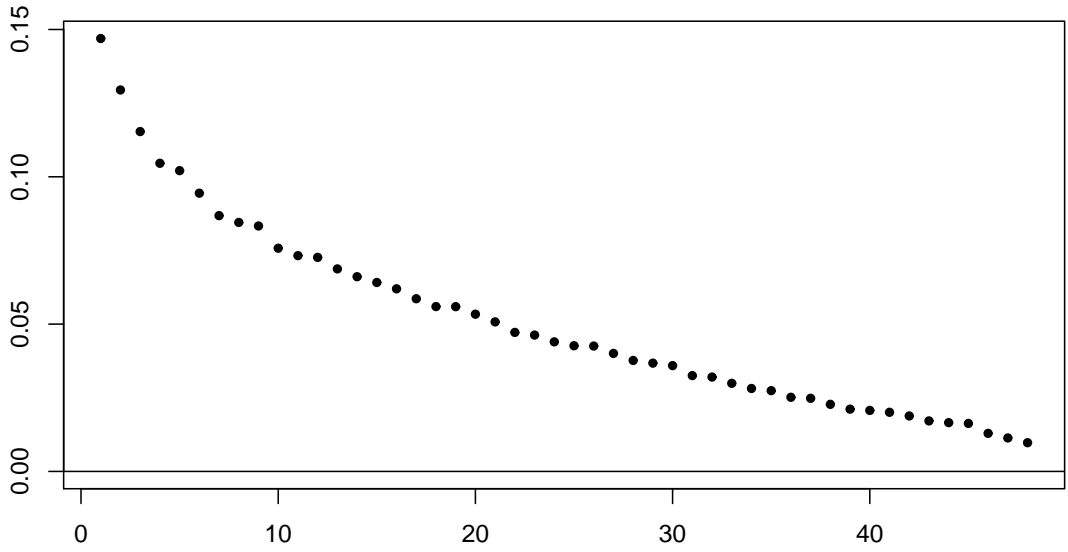}
\caption{Melbourne temperature data:
eigenvalues $\hat{\lambda}_2,\ldots,\hat{\lambda}_{49}$.}
\label{fig:3b}
\end{figure}

\begin{figure}
 \centering
 \includegraphics[scale = 0.775,angle = 270]{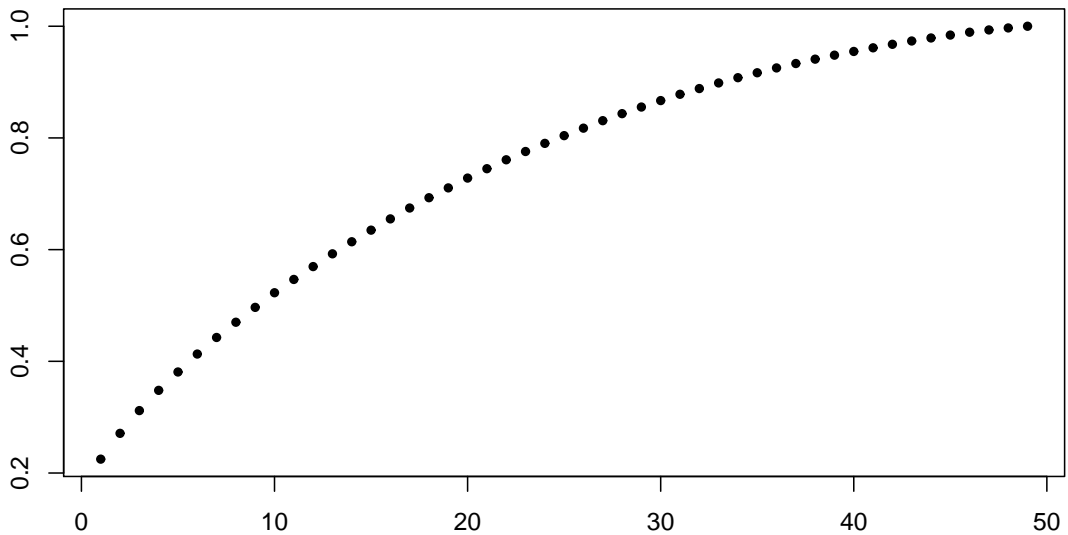}
\caption{Melbourne temperature data:
percentage of variance explained by the first $k$ eigenvalues,
 i.e.
$f_k = \sum_{i = 1}^k\hat{\lambda}_i/\sum_{j = 1}^N\hat{\lambda}_j,
k = 1,2,\ldots,49$. }
\label{fig:5}
\end{figure}

Statistical procedures for functional data which use
functional principal components (FPC's) often
depend on the number $d$ of the components used to compute various
statistics. The selection of an optimal $d$ has received a fair deal of
attention. Commonly used approaches include the cumulative variance
method, the scree plot, and several forms of cross--validation and
pseudo information criteria.  By now, most of these approaches are
implemented in several R packages and in the Matlab package PACE.  A
related direction of research has focused on the identification of the
dimension $d$ assuming that the functional data actually live in a
finite\nobreakdash--dimensional space of this dimension, see \citet{hall:vial:2006}
and \citet{bathia:yao:ziegelmann:2010}.  The research presented in
this paper is concerned with functional data which cannot be reduced
to finite--dimensional data in an obvious and easy way.
Such data are typically characterized by a slow decay of the
eigenvalues of the empirical covariance operator.
Figure~\ref{fig:3b} shows the eigenvalues of the
empirical covariance operator of the annual temperature curves
obtained over the period 1856--2011 in Melbourne, Australia,
while Figure~\ref{fig:5} shows the cumulative variance plot for the
same data set. It is seen that the eigenfunctions decay at a slow rate,
and neither their  visual inspection nor the analysis of cumulative variance
provide a clear guidance on how to select $d$. This data set is analyzed
in greater detail in Section~\ref{s:stefan}.

In situations when the choice of $d$ is difficult, two approaches seem
reasonable. In the first approach, one can apply a test using several
values of $d$ in a reasonable range. If the conclusion does not depend
on $d$, we can be confident that it is correct.  This approach has
been used in applied research, see \citet{gromenko:kokoszka:2012big}
for a recent analysis of this type.  The second approach, would be to
let $d$ increase with the sample size $N$, and derive a test statistic
based on the limit.  In a sense, the second approach is a
formalization of the first one because if a limit as $d\to \infty$
exists, then the conclusions should not depend on the choice of $d$,
if it is reasonably large.  In the FDA community there is a well
grounded intuition that $d$ should increase much slower than $N$, so
asymptotically large $d$ need not be very large in practice.  It is
also known that the rate at which $d$ increases should depend on the
manner  in  which the eigenvalues decay. We obtain specific conditions
that formalize this intuition in the framework we consider.  In more
specific settings, contributions in this directions were made by
\citet{cardot:fms:2003} and \citet{panaretos:2010}. The work of
\citet{cardot:fms:2003} is more closely related to our research: as
part of the justification of their testing procedure, they
establish conditions under which a  limiting chi--square distribution
with $d$ degrees of freedom can be approximated by a normal
distribution as $d= d(N)\to \infty$.  \citet{panaretos:2010} are
concerned with a test of the equality of the covariance operators in
two samples of Gaussian curves. In the supplemental material, they
derive asymptotics in which $d$ is allowed to increase with the sample
size.  Our theory is geared toward testing the equality of mean
functions, but we do not assume the normality of the functional
observations, so we cannot use arguments that use the equivalence of
independence and zero covariances.  We develop a new technique based
on the estimation of the Prokhorov--L{\'e}vy distance between the
underlying processes and the corresponding normal partial sums.

The paper is organized as follows. In Section~\ref{s:approx},
we set the framework and state a general normal approximation result
in Theorem~\ref{approx}.  This result is then used in Sections \ref{s:cp}
and \ref{s:two-s} to derive, respectively,  change--point and
two--sample tests based on an increasing
number of FPC's. Section~\ref{s:stefan} contains a small simulation
study and an application to the annual Melbourne
temperature curves. All proofs are collected in the appendices.

\section{\bf  Uniform normal
approximation}
\label{s:approx}

We consider  functional
observations $X_i(t),\ t\in {\mathcal I},\  i=1,2,\ldots, N,$
defined over a compact interval
${\mathcal I}$. We can and shall assume without loss of generality
that ${\mathcal I}=[0,1]$.
Throughout the paper,  we use the notation $\int =\int_0^1$ and
\[
\langle f, g \rangle = \int f(t) g(t) dt, \ \ \
\norm f\norm^2= \langle f, f \rangle.
\]
All functions we consider will be elements of the Hilbert space $L^2$
of square integrable functions on $[0,1]$.

In the testing problems that motivate this research, under the null
hypothesis, the observations follow the model
\begin{equation} \label{null-model}
X_i(t)=\mu(t)+Z_i(t),\;\;\;1\leq i \leq N,
\end{equation}
where $EZ_i(t)=0$ and  $\mu(t)$ is the common mean.
We impose the following standard assumptions.

\begin{assumption}\label{as-1} $Z_1, Z_2,\ldots, Z_N$
are independent and identically distributed.
\end{assumption}

\begin{assumption}\label{as-2} $\int\mu^2(t)dt<\infty$
and $E\norm Z_1\norm^2<\infty.$
\end{assumption}

Under these assumptions, the covariance function
\[
{\mathfrak c}(t,s)=EZ_1(t)Z_1(s),
\]
is square integrable on the unit square and therefore it
has the representation
\[
{\mathfrak c}(t,s)=\sum_{k=1}^\infty \lambda_kv_k(t)v_k(s),
\]
where $\lambda_1\geq \lambda_2\geq \ldots$ are the eigenvalues and
$v_1, v_2,\ldots$ are the orthonormal eigenfunctions of the covariance
operator, i.e. they satisfy the integral equation
\beq\label{eq-eig}
\lambda_j v_j(t)=\int {\mathfrak c}(t,s)v_j(s)ds.
\eeq

One of the most important dimension reduction techniques
of functional data analysis is to project the
observations $X_1(t), \ldots, X_N(t)$ onto the space
spanned by $v_1, \ldots ,v_d$, the eigenfunctions associated with the
$d$ largest eigenvalues. Since the covariance
function ${\mathfrak c}$, and therefore $v_1, \ldots ,v_d$,
are unknown,  we use the empirical eigenfunctions
$\hat{v}_1, \ldots ,\hat{v}_d$ and eigenvalues $\hat{\lambda}_1\geq
\hat{\lambda}_2\geq \ldots \geq \hat{\lambda}_d$ defined by
\beq\label{eq-eigemp}
\hat{\lambda}_j \hat{v}_j(t)=\int \hat{\mathfrak c}_N(t,s)\hat{v}_j(s)ds,
\eeq
where
\[
\hat{\mathfrak c}_N(t,s)=\frac{1}{N}
\sum_{i=1}^N \left ( X_i(t) - \bar{X}_N(t) \right )
\left ( X_i(s) - \bar{X}_N(s) \right )
\]
with
$
\bar{X}_N(t)=N^{-1}\sum_{i=1}^NX_i(t).
$

In this section, we require only two more assumptions, namely

\begin{assumption}\label{as-5}\;\;\;\;\quad
$\lambda_1>\lambda_2>\ldots$
\end{assumption}

\begin{assumption}\label{as-6}\;\;\;\;\quad
$E\norm Z_1\norm^3<\infty.$
\end{assumption}

Assumption~\ref{as-5} is needed to ensure that the FPC's $v_j$ are
uniquely defined. In Theorem~\ref{approx} it could, of course, be
replaced by requiring only that the first $d$ eigenvalues are positive
and different, but since in the applications we let $d\to \infty$,
we just assume that all eigenvalues are positive and distinct. If
$\lambda_{d^*+1}=0$ for some $d^*$, then the observations are in the
linear span of $v_1, \ldots ,v_{d^*}$, i.e. they are elements of a
$d^*$--dimensional space, so in this case we cannot consider
$d=d(N)\to \infty$. Assumption \ref{as-5} means that the observations
are in an infinite--dimensional space.  Assumption~\ref{as-6} is
weaker than the usual assumption $E\norm Z_1\norm^4<\infty$. As will
be seen in the proofs, subtle arguments of the probability theory in
Banach spaces are needed to dispense with the fourth moment.

To state the main result of this section, define
\begin{align*}
\bxi_i=(\xi_{i,1}, \ldots ,\xi_{i,d})^T\;\;\;\mbox{and}\;\;\;\xi_{i,j}
=\lambda_j^{-1/2}\la Z_i, v_j\ra,\;\;1\leq i \leq N,\; 1\leq j \leq d,
\end{align*}
where $\cdot^T$ denotes the transpose of vectors and matrices.
Set
\beq\label{n-sum}
S_{j,N}(x)=\frac{1}{N^{1/2}}\sum_{i=1}^{\lfloor Nx\rfloor}\xi_{i,j},
\ \ \ \ 0\le x \le 1, \ \ 1\le j \le d.
\eeq

We now provide an approximation  for the partial sum processes
$S_{j,N}(x)$ defined in (\ref{n-sum})
with suitably constructed Wiener processes (standard Brownian motions).
\begin{theorem}\label{approx}
 If Assumptions \ref{as-1}, \ref{as-5}
and \ref{as-6}  hold, then  for every $N$ we can define independent
Wiener processes $W_{1,N}, \ldots ,W_{d, N}$ such that
\begin{align}\label{main}
P\biggl\{\max_{1\leq j\leq d}\;\sup_{0\leq x \leq 1}
&\left |S_{j,N}(x)-W_{j,N}(x)\right |\geq N^{1/2-1/80} \biggl\}\\
&\leq c_*N^{-1/80}\biggl\{d^{1/12}
\biggl(\sum_{\ell=1}^d1/\lambda_\ell\biggl)^{1/8}
+\sum_{j=1}^d1/\lambda^{3/2}_j
\biggl\},\notag
\end{align}
where $c_*$ only depends on $\lambda_1$ and $E\norm Z_1\norm^3.$
\end{theorem}

The constant $1/80$ in \eqref{main} is not crucial, it is a result
of our calculations.
Theorem~\ref{approx} is related to the results of
Einmahl (\citeyear{einmahl:1987}, \citeyear{einmahl:1989}) who
obtained strong approximations for partial sums of independent and
identically distributed random vectors with zero mean and
with identity covariance matrix. In our setting, for any fixed $d$,
the covariance matrix is not the identity, but this is not the central  difficulty.
The main value of Theorem~\ref{approx} stems from the fact that
it shows how the rate of the approximation depends on $d$;
no such information is contained in the work of
Einmahl (\citeyear{einmahl:1987}, \citeyear{einmahl:1989}),
who did not need to consider the dependence on $d$. The explicit
dependence of the right hand side of \eqref{main} on $d$ is crucial
in the applications presented in the following sections in which the dimension
of the projection space depends on the sample size $N$.

Very broadly speaking,
Theorem~\ref{approx} implies  that in all reasonable
statistics  based on averaging the scores,
even in those based on an increasing number
of FPC's, the partial sums of scores
can be replaced by Wiener processes to obtain a limit distribution.
The right hand side of \eqref{main} allows us to derive assumptions
on the eigenvalues required to obtain a specific result. Replacing
the unobservable scores $\xi_{i,j}$ by the sample scores
$\hat{\eta}_{i,j}$ is relatively easy. We will illustrate these ideas
in Sections \ref{s:cp} and \ref{s:two-s}.

\section{\bf  Change--point detection}
\label{s:cp}
Over the past four decades,
the investigation of the asymptotic properties of partial sum
processes has to a large extent been motivated by
change--point detection procedures, and this is the most natural
application of Theorem~\ref{approx}.  The research on the change--point
problem in various contexts is very extensive, some aspects
of the asymptotic theory are presented in \citet{csorgo:horvath:1997}.
Detection of a change in the mean function was studied
by \citet{berkes:gabrys:horvath:kokoszka:2009} who considered
a procedure in which the number of the FPC's, $d$,  was fixed, and the asymptotic distribution of the test statistic depended on $d$.
We show in this section that it is possible to derive tests with
a standard normal limiting distribution by allowing the $d$ to depend
on the sample size $N$.

We want to test whether the mean of the
observations remained the same during the observation period, i.e. we
test the null hypothesis
\[
H_0:\;\;EX_1(\cdot)=EX_2(\cdot)=\cdots =EX_N(\cdot)
\]
(``='' means equality in $L^2$). Under the null hypothesis, the $X_i$
follow model (\ref{null-model}) in which
$\mu(\cdot)$ is an unknown common mean function under $H_0$.
The alternative hypothesis is
\begin{align*}
H_A:\;&\mbox{there is}\;\;k^*\in [1,2,\ldots, N)\;\mbox{ such that}\\
&EX_1(\cdot)=\cdots=EX_{k^*}(\cdot)\neq EX_{k^*+1}(\cdot)=\cdots
=EX_N(\cdot).
\end{align*}
Under $H_A$ the mean changes at an  unknown time $k^*$.

To derive a new class of tests, we introduce the process
\[
\hat{Z}_N(u,x)
=\frac{1}{d^{1/2}}\sum_{j=1}^{\lfloor du\rfloor}
\left\{ \frac{1}{N}
\left [ \hat{S}_{j}( \lfloor Nx\rfloor ) -x\hat{S}_j(N)\right ]^2
-x(1-x)\right\},\;\;0\leq u, x \leq 1,
\]
where
\[
\hat{S}_j(k)=\frac{1}{\hat{\lambda}_j^{1/2}}
\sum_{i=1}^k\hat{\eta}_{i,j}.
\]
The process $\hat{Z}_N(u,x)$ contains the cumulative sums
$\hat{S}_{j}( \lfloor Nx\rfloor ) -x\hat{S}_j(N)$ which measure the
deviation of the partial sums from their ``trend'' under $H_0$,
and a correction term $x(1-x)$ needed to ensure convergence
as $d\to \infty$.

To obtain a limit which does not depend on any unknown
quantities, we need to impose assumptions on the rate at
which $d=d(N)$ increases with $N$. Intuitively, the
assumptions below state that  $d$ is much smaller than
the sample size $N$,  the $d$ largest eigenvalues are not too small,
and that the difference between the
consecutive eigenvalues tends to zero slowly.  Very broadly speaking,
these assumptions mean that the distribution of the observations
must sufficiently fill the whole infinite--dimensional space $L^2$.

\begin{assumption}\label{d-0}\quad $d=d(N)\to \infty$ \end{assumption}
\begin{assumption}\label{d-1}\quad
$ (d\log N)^{1/2}N^{-1/80}\to 0,
$
\end{assumption}
\begin{assumption}\label{d-2}\quad
$ d^{1/12}N^{-1/80}\biggl(\displaystyle
\sum_{j=1}^d\displaystyle 1/\lambda_j\biggl)^{1/8}\to 0.
$
\end{assumption}
\begin{assumption}\label{d-3}\quad
$
 N^{-1/80}\displaystyle\sum_{j=1}^d
\displaystyle 1/\lambda_j^{3/2}\to 0.
$
\end{assumption}
\begin{assumption}\label{d-4}
$$
\frac{1}{d^{1/2}N^{1/3}}\sum_{j=1}^d\frac{1}{\lambda_j\zeta_j}\to 0,
$$
where $\zeta_1=\lambda_2-\lambda_1$,
$\zeta_j=\min(\lambda_{j-1}-\lambda_j, \lambda_j-\lambda_{j+1}),\ j\geq 2$.
\end{assumption}

With these preparations, we can state the main result of this section.

\begin{theorem}\label{th-1} If Assumptions \ref{as-1}--\ref{as-5}
and \ref{d-0}--\ref{d-4} are satisfied, then
\[
\hat{Z}_N(u,x)\;\to\;\;\Gamma(u,x)\;\;\mbox{in}\;\;{\mathcal D}[0,1]^2,
\]
where  $\Gamma(u,x)$ is a mean zero Gaussian process with
\[
E[\Gamma(u,x)\Gamma(v,y)]=2ux^2(1-y)^2, \ \ \  0\leq u\leq v \leq 1, \ \
0\leq x\leq y  \leq 1.
\]
\end{theorem}

One can  verify by computing the covariance functions that
\begin{equation} \label{e:rep}
\{\Gamma(u,x),\; 0\leq u,x\leq 1\}
\stackrel{{\mathcal D}}{=}\{\sqrt{2}(1-x)^2W(u,x^2/(1-x)^2),
\; 0\leq u,x\leq 1\},
\end{equation}
where $\{W(v,y), v,y\geq 0\}$ is a bivariate  Wiener process,
i.e. $W(v,y)$ is a Gaussian process with
$EW(v,y)=0$ and $E[W(v,y)W(v',y')]=\min(v,v')\min(y,y')$.
Representation \eqref{e:rep} means that continuous functionals
of the process $\Gamma(\cdot, \cdot)$ can be simulated
with arbitrary precision, so  Monte Carlo tests can be used.
One would choose the number of projections in the CUSUM 
procedure such that the test would give
the largest rejection if the alternative holds.
 The statistic $\max_u\max_x|\hat{Z}_N(u,x)|$ is maximizing the CUSUM
statistics $\max_x|\hat{Z}_N(k/d,x)|$, where $k=1,2,\ldots ,d$ projections are used.
It is however possible to obtain a number of simple asymptotic tests
by examining closer the structure of the process  $\Gamma(\cdot, \cdot)$.
We list some of them in Corollary~\ref{col}, and we will see
in Section~\ref{s:stefan} that the Cram{\'e}r-von-Mises type tests
have very good finite sample properties.  Let $B$ denote a
Brownian bridge and define
\[
\mu_0=E\left(\sup_{0\leq x \leq 1}B^2(x)\right) \ \ \ {\rm and} \ \ \
\sigma_0^2=\mbox{var}\left (\sup_{0\leq x \leq 1}B^2(x)\right ).
\]

\begin{corollary}\label{col}
If the assumptions of Theorem \ref{th-1} are satisfied, then
\begin{equation} \label{col-1}
\frac{1}{d^{1/2}\sigma_0}\left\{\sum_{j=1}^{d}
\sup_{0\leq x\leq 1}\frac{1}{N}\left (\hat{S}_{j}( \lfloor Nx\rfloor )
-x\hat{S}_j(N)\right)^2
-d\mu_0\right\}\;\;\stackrel{{\mathcal D}}{\to}\;\;N(0,1),
\end{equation}
\begin{equation} \label{col-2}
\frac{1}{(d/45)^{1/2}}\left\{\sum_{j=1}^{d}
\frac{1}{N}\int (\hat{S}_{j}( \lfloor Nx\rfloor )-x\hat{S}_j(N))^2dx
-\frac{d}{6}\right\}\;\;\stackrel{{\mathcal D}}{\to}\;\;N(0,1),
\end{equation}
\begin{equation} \label{col-3}
\frac{1}{(d/8)^{1/2}}\left\{\sup_{0\leq x \leq 1}\sum_{j=1}^{d}
\frac{1}{N} (\hat{S}_{j}( \lfloor Nx\rfloor )-x\hat{S}_j(N))^2
-\frac{d}{4}\right\}\;\;\stackrel{{\mathcal D}}{\to}\;\;N(0,1),
\end{equation}
where $N(0,1)$ stands for a standard normal random variable.
\end{corollary}

We conclude this section with two examples which show
that Assumptions \ref{d-1}--\ref{d-4} hold under both power law
and exponential decay of the eigenvalues.

\begin{example}\label{ex-1} If the eigenvalues satisfy
$$
\lambda_j=\frac{c_1}{(j-c_2)^\alpha}+o\left(\frac{1}{j^{\alpha+1}}\right),
\;\;\;\mbox{as}\;\;\;j\to \infty,
$$
with some $c_1>0, \;0\leq c_2<1$ and $\alpha>0$, then Assumptions
\ref{d-1}--\ref{d-4} hold if $d/(\log N)^\beta\to 0$ with some $\beta>0$.
\end{example}

Under the conditions of Example 3.1, one could choose $d_N=O(N^\zeta)$, 
where $\zeta$
depends on $\alpha$. In case of a fixed samplle size $N$,
the  power of the test would decrease if $d_N$ is too large.
 Hence we recommend choosing $d_N\approx (\log N)^\beta$, 
 where $\beta>0$ can be arbitrarily chosen.

\begin{example}\label{ex-2} If the eigenvalues satisfy
$$
\lambda_j=c_0e^{-\alpha j}+o(e^{-\alpha j}),\;\;\;\mbox{as}\;\;\;j\to \infty,
$$
with some $c_0>0$ and $\alpha>0$, then Assumptions \ref{d-1}--\ref{d-4}
hold if $d/(\log\log N)^\beta\to 0$ with some $\beta>0$.
\end{example}

\section{\bf Two--sample problem}
 \label{s:two-s}
The two--sample problem for functional data was perhaps first
discussed in depth by \citet{benko:hardle:kneip:2009} who
were motivated by a problem related to implied volatility curves.
It has recently attracted a fair amount of attention motivated
by problems arising in space physics, see
\citet{horvath:kokoszka:reimherr:2009},
genetics, see \citet{panaretos:2010}, and finance, see
\citet{horvath:kokoszka:reeder:2012}. The above list does not
include many other important contributions.
In its simplest, but most
important form, it is about testing if curves obtained from two
populations have the same mean functions. The most direct
approach, developed into a bootstrap procedure by
\citet{benko:hardle:kneip:2009}, is to look at the norm of the difference
of the estimated mean functions. In this section, we show that
the normal approximation of Section~\ref{s:approx} leads to an
asymptotic test whose  limit  distribution is standard normal.

Suppose we have two random samples of
functions: $X_1,\ldots,X_N$ and $Y_1,\ldots,
Y_M$. We assume the $X$ sample satisfies (\ref{null-model}) and
Assumptions \ref{as-1}, \ref{as-2} and \ref{as-6}. Similarly, the $Y$
sample is a location model given by
\beq\label{y-model}
Y_i(t)=\mu_*(t)+Q_i(t),\;\;\;1\leq i \leq M,
\eeq
where $\mu_*(t)$ is the common mean of the $Y$ sample and
$EQ_i(t)=0$.
As in the case of the $X$ sample, the $Y$ sample satisfies the following
conditions:
\begin{assumption}\label{as-1-v} $Q_1, Q_2,\ldots, Q_M$
are independent and identically distributed.
\end{assumption}
\begin{assumption}\label{as-2-v} $\int\mu_*^2(t)dt<\infty$
and $E\norm Q_1\norm^3<\infty$.
\end{assumption}

Assumption \ref{as-2-v} yields that
$$
{\mathfrak c}_*(t,s)=EQ_1(t)Q_1(s)
$$
is a square integrable function on the unit square.\\
In this section we are interested in testing the null hypothesis
$$
H_0^*:\;\;\mu(\cdot)=\mu_*(\cdot).
$$
The statistical inference to test $H_0$ is based on
the difference $\bar{X}_N-\bar{Y}_M$,
where $\bar{X}_N$ and $\bar{Y}_M$ denote the sample means.
We assume
\begin{assumption}\label{m-1}
$$
\frac{N}{M}=\lambda +O(N^{-1/4})\;\;\;\mbox{as}\;\;\min(M,N)\to\infty
$$
with some $0<\lambda <\infty$.
\end{assumption}
Now we define the pooled covariance function
\[
{\mathfrak c}_P(t,s)={\mathfrak c}(t,s)+\lambda {\mathfrak c}_*(t,s).
\]
Since ${\mathfrak c}_P(t,s)$ is a positive--definite, symmetric, square
integrable function,  there are real
numbers $\kappa_1\geq \kappa_2\geq \ldots $ and orthonormal
functions $u_1, u_2,\ldots$ satisfying
\[
\kappa_i u_i(t)=\int {\mathfrak c}_P(t,s)u_i(s)ds,\;\;i=1,2,\ldots.
\]
We wish to project $\bar{X}_N-\bar{Y}_M$ into the space spanned by
$u_1,\ldots ,u_d$, where $d=d(N)\to \infty$, so similarly to
Assumption \ref{as-5} we require
\begin{assumption}\label{m-2}\;\;$\kappa_1> \kappa_2>\kappa_3>\ldots$
\end{assumption}
\begin{assumption}\label{m-3}
$$
N^{-3/32}d^{1/4}\left(\sum_{\ell=1}^d 1/\kappa_\ell\right)^{3/8}\to 0.
$$
\end{assumption}
Our test statistic is
$$
D_{N,M}=\sum_{i=1}^dN\la \bar{X}_N-\bar{Y}_M, u_i\ra^2/\kappa_i.
$$
As in Section \ref{s:cp},
we need additional assumptions balancing
the rate of growth of $d=d(N)$ and the rate of decay of the $\kappa_\ell$
and the differences between them.
\begin{assumption}\label{m-5}
\begin{align*}
\frac{1}{d^{1/2}N^{1/4}}\sum_{\ell=1}^d\frac{1}{\kappa_\ell^2}
\to 0\;\;\;\;\mbox{and}\;\;\;\;
 \frac{1}{d^{1/2}N^{1/4}}\sum_{\ell=1}^d\frac{1}{\kappa_\ell \iota_\ell}\to 0,
\end{align*}
where $\iota_1=\kappa_2-\kappa_1,\ \iota_\ell=\min(\iota_{\ell-1}-\iota_\ell, \iota_{\ell}-\iota_{\ell+1}),\
\ell\geq 2$.
\end{assumption}

Since $u_1,u_2,\ldots $ are unknown, we replace them with the
corresponding empirical eigenfunctions $\hat{u}_1, \hat{u}_2,\ldots $
defined by the integral operator
\[
\hat{\kappa}_i \hat{u}_i(t)
=\int \hat{\mathfrak c}_P(t,s)\hat{u}_i(s)ds,\;\;\ i=1,2,\ldots,
\]
where $\hat{\kappa}_1\geq \hat{\kappa}_2\geq \ldots$ and
$$
\hat{\mathfrak c}_P(t,s)=\hat{\mathfrak c}_N(t,s)
+\frac{N}{M}\hat{\mathfrak c}_{*M}(t,s),
$$
with
$$
\hat{\mathfrak c}_{*M}(t,s)
=\frac{1}{M}\sum_{\ell=1}^M(Y_\ell(t)-\bar{Y}_M(t))(Y_\ell(s)-\bar{Y}_M(s)).
$$
The empirical version of $D_{N,M}$ is
$$
\widehat{D}_{N,M}
=\sum_{i=1}^dN\la \bar{X}_N-\bar{Y}_M, \hat{u}_i\ra^2/\hat{\kappa}_i.
$$

\begin{theorem}\label{th-two-1} If $H^*_0$,
Assumptions \ref{as-1}, \ref{as-2} and \ref{as-1-v}--\ref{m-5}
hold, then
\[
(2d)^{-1/2}(\widehat{D}_{N,M}-d)\;\;\;
\stackrel{{\mathcal D}}{\to}\;\;\;N(0,1),
\]
where $N(0,1)$ stands for a  standard normal random variable.
\end{theorem}

\section{\bf A small simulation
study and a data example} \label{s:stefan}
\setcounter{equation}{0}
The main contribution of this paper lies in the statistical theory, but it is
of interest to check if the new tests derived in Sections \ref{s:cp}
and \ref{s:two-s} perform well in finite samples. We report the results
for the test based on Theorem~\ref{th-1} in some detail, as it
utilizes the convergence of the two--parameter process in full force, and
such an approach has not been used before. We also comment
on the tests based on Corollary~\ref{col} and Theorem~\ref{th-two-1}.
We conclude this section with an illustrative  data example.

The simulated data which satisfy the null hypotheses of Sections \ref{s:cp}
and \ref{s:two-s} are generated as independent Brownian motions
on the interval $[0,1]$. We generate them by using iid normal increments
on 1,000 equispaced points in $[0,1]$ (random walk approximation).
(Example \ref{ex-1} shows that for the Brownian motion the assumptions of
Theorem \ref{th-1} are satisfied.)
Alternatives are obtained by
adding the curve $at(1-t)$ after a change--point or to the observations
in the second sample. The parameter $a$ regulates the size of the change
or the difference in the means in two samples.

\begin{table}
\centering
\begin{tabular}{cccc}
\toprule
$\alpha$ & 0.01 & 0.05 & 0.10\\
\cmidrule(lr){1-4}
& 0.109256 & 0.0726292 & 0.0578267\\
\bottomrule
\end{tabular}
\bigskip
\caption{Critical values for the distribution of \eqref{TLD}.}\label{CVT}
\end{table}

\begin{table}
\centering
\begin{tabular}{ccccc}
$N = 100$&&&&\\
\midrule
 &  \multicolumn{2}{c}{$\alpha = 0.05$} & \multicolumn{2}{c}{$\alpha = 0.1$}\\
\cmidrule(lr){2-3}\cmidrule(lr){4-5}
$d$ & $\hat{p}$ & $[a,b]$ & $\hat{p}$ & $[a,b]$\\
\cmidrule(lr){1-1}\cmidrule(lr){2-2}\cmidrule(lr){3-3}\cmidrule(lr){4-4}\cmidrule(lr){5-5}
2 & 0.047 & [0.0360,0.0580] & 0.058 & [0.0458,0.0702]\\
3 & 0.056 & [0.0440,0.0680] & 0.074 & [0.0604,0.0876]\\
4 & 0.060 & [0.0476,0.0724] & 0.081 & [0.0668,0.0952]\\
5 & 0.059 & [0.0467,0.0713] & 0.089 & [0.0742,0.1038]\\
6 & 0.057 & [0.0449,0.0691] & 0.089 & [0.0742,0.1038]\\
7 & 0.056 & [0.0440,0.0680] & 0.089 & [0.0742,0.1038]\\
8 & 0.059 & [0.0467,0.0713] & 0.091 & [0.0760,0.1060]\\
9 & 0.051 & [0.0396,0.0624] & 0.090 & [0.0751,0.1049]\\
10 & 0.050 & [0.0387,0.0613] & 0.082 & [0.0677,0.0963]\\
11 & 0.054 & [0.0422,0.0658] & 0.083 & [0.0687,0.0973]\\
12 & 0.057 & [0.0449,0.0691] & 0.079 & [0.0650,0.0930]\\
13 & 0.059 & [0.0467,0.0713] & 0.075 & [0.0613,0.0887]\\
14 & 0.057 & [0.0449,0.0691] & 0.076 & [0.0622,0.0898]\\
15 & 0.056 & [0.0440,0.0680] & 0.075 & [0.0613,0.0887]\\
&&&&\\
$N = 200$&&&&\\
\midrule
 &  \multicolumn{2}{c}{$\alpha = 0.05$} & \multicolumn{2}{c}{$\alpha = 0.1$}\\
\cmidrule(lr){2-3}\cmidrule(lr){4-5}
$d$ & $\hat{p}$ & $[a,b]$ & $\hat{p}$ & $[a,b]$\\
\cmidrule(lr){1-1}\cmidrule(lr){2-2}\cmidrule(lr){3-3}\cmidrule(lr){4-4}\cmidrule(lr){5-5}
2 & 0.039 & [0.0289,0.0491] & 0.055 & [0.0431,0.0669]\\
3 & 0.048 & [0.0369,0.0591] & 0.070 & [0.0567,0.0833]\\
4 & 0.049 & [0.0378,0.0602] & 0.075 & [0.0613,0.0887]\\
5 & 0.053 & [0.0413,0.0647] & 0.076 & [0.0622,0.0898]\\
6 & 0.057 & [0.0449,0.0691] & 0.085 & [0.0705,0.0995\\
7 & 0.057 & [0.0449,0.0691] & 0.085 & [0.0705,0.0995]\\
8 & 0.053 & [0.0413,0.0647] & 0.085 & [0.0705,0.0995]\\
9 & 0.051 & [0.0396,0.0624] & 0.083 & [0.0687,0.0973]\\
10 & 0.051 & [0.0378,0.0602] & 0.081 & [0.0668,0.0952]\\
11 & 0.054 & [0.0496,0.0624] & 0.083 & [0.0687,0.0973]\\
12 & 0.052 & [0.0405,0.0635] & 0.086 & [0.0714,0.1006]\\
13 & 0.050 & [0.0387,0.0613] & 0.087 & [0.0723,0.1017]\\
14 & 0.054 & [0.0422,0.0658] & 0.086 & [0.0714,0.1006]\\
15 & 0.052 & [0.0405,0.0635] & 0.079 & [0.0650,0.0930]\\
\end{tabular}
 \bigskip
\caption{Empirical sizes and 90\% confidence intervals
for the probability of rejection  for  the change--point test based on
 convergence \eqref{e:wc}.
} \label{ES-CI}
\end{table}

Many tests can be obtained from Theorem~\ref{th-1} by applying
functionals continuous on ${\mathcal D}[0,1]^2$. It is not our objective
to provide a systematic comparison, we consider only the test based on the
weak convergence
\begin{equation} \label{e:wc}
\int_0^1\int_0^1 \hat{Z}^2_N(u,x)dudx\;\to\;\;
\int_0^1\int_0^1\Gamma^2(u,x) dudx.
\end{equation}
To compute the critical values,
we use the following representation of the limit
\begin{align}
 \int_0^1\int_0^1 \Gamma^2(u,x) du dx
&\stackrel{\mathcal{D}}{=}
\sum_{1\leq k,\ell<\infty}\lambda_k \nu_\ell N_{k,\ell}^2.\label{LD}
\end{align}
In \eqref{LD}, the
$\lambda_k = (\pi (k-1/2))^{-2}$ are the eigenvalues of the Wiener
process, the $\nu_\ell$ are the eigenvalues of the covariance operator
with kernel $2(\min(s,t)-st)^2$,  and $\{N_{k,\ell}\}$ is an
array of independent standard normal random variables. The critical
values were determined for a truncated version of the right--hand side
of \eqref{LD} with truncation level 49, i.e. for
\begin{align}
\sum_{1\leq k,\ell\leq 49}\lambda_k \nu_\ell N_{k,\ell}^2.\label{TLD}
\end{align}
Since the eigenvalues $\nu_\ell$ are difficult to determine
explicitly, they were calculated numerically using the \texttt{R}
package \texttt{fda}, cf.\citet{ramsay:hooker:graves:2009}.  The
simulated critical values based on 100,000 replications of \eqref{TLD}
are provided in Table~\ref{CVT}.

\begin{table}
\centering
\begin{tabular}{ccccc}
$N = 100$&&&&\\
\midrule
 &  \multicolumn{2}{c}{$a = 1$} & \multicolumn{2}{c}{$a = 1.5$}\\
\cmidrule(lr){2-3}\cmidrule(lr){4-5}
$d$ & $\alpha = 0.05$ & $\alpha = 0.10$ & $\alpha = 0.05$ & $\alpha = 0.10$\\
\cmidrule(lr){1-1}\cmidrule(lr){2-2}\cmidrule(lr){3-3}\cmidrule(lr){4-4}\cmidrule(lr){5-5}
2 & 0.168 & 0.192 & 0.356 & 0.398 \\
3 & 0.456 & 0.517 & 0.819 & 0.851 \\
4 & 0.501 & 0.564 & 0.843 & 0.875 \\
5 & 0.496 & 0.564 & 0.855 & 0.887 \\
6 & 0.481 & 0.552 & 0.847 & 0.883 \\
7 & 0.473 & 0.543 & 0.843 & 0.881 \\
8 & 0.465 & 0.530 & 0.834 & 0.874 \\
9 & 0.461 & 0.519 & 0.823 & 0.870 \\
10 & 0.453 & 0.504 & 0.812 & 0.859 \\
11 & 0.441 & 0.501 & 0.802 & 0.853 \\
12 & 0.431 & 0.496 & 0.793 & 0.844 \\
13 & 0.420 & 0.484 & 0.791 & 0.834 \\
14 & 0.400 & 0.472 & 0.782 & 0.822 \\
15 & 0.388 & 0.467 & 0.767 & 0.817 \\
&&&&\\
$N = 200$&&&&\\
\midrule
 &  \multicolumn{2}{c}{$a = 1$} & \multicolumn{2}{c}{$a = 1.5$}\\
\cmidrule(lr){2-3}\cmidrule(lr){4-5}
$d$ & $\alpha = 0.05$ & $\alpha = 0.10$ & $\alpha = 0.05$ & $\alpha = 0.10$\\
\cmidrule(lr){1-1}\cmidrule(lr){2-2}\cmidrule(lr){3-3}\cmidrule(lr){4-4}\cmidrule(lr){5-5}
2 & 0.327 & 0.370 & 0.620  & 0.660 \\
3 & 0.784 & 0.814 & 0.984 & 0.991 \\
4 & 0.808 & 0.849 & 0.988 & 0.994 \\
5 & 0.823 & 0.860 & 0.992 & 0.994 \\
6 & 0.825 & 0.863 & 0.991 & 0.996 \\
7 & 0.819 & 0.864 & 0.992 & 0.994 \\
8 & 0.814 & 0.859 & 0.990 & 0.994 \\
9 & 0.802 & 0.846 & 0.990 & 0.993 \\
10 & 0.791 & 0.837 & 0.990 & 0.993 \\
11 & 0.766 & 0.830 & 0.988 & 0.992 \\
12 & 0.754 & 0.821 & 0.987 & 0.992 \\
13 & 0.740 & 0.800 & 0.987 & 0.991 \\
14 & 0.734 & 0.794 & 0.987 & 0.991 \\
15 & 0.726 & 0.787 & 0.986 & 0.990 \\
\end{tabular}
 \bigskip
 \caption{Power of the  test based on convergence
\eqref{e:wc}. The change--point is at $k^* = \lfloor N/2\rfloor$.}
 \label{power}
\end{table}

\begin{figure}
 \centering
 \includegraphics[scale = 0.775,angle = 270]{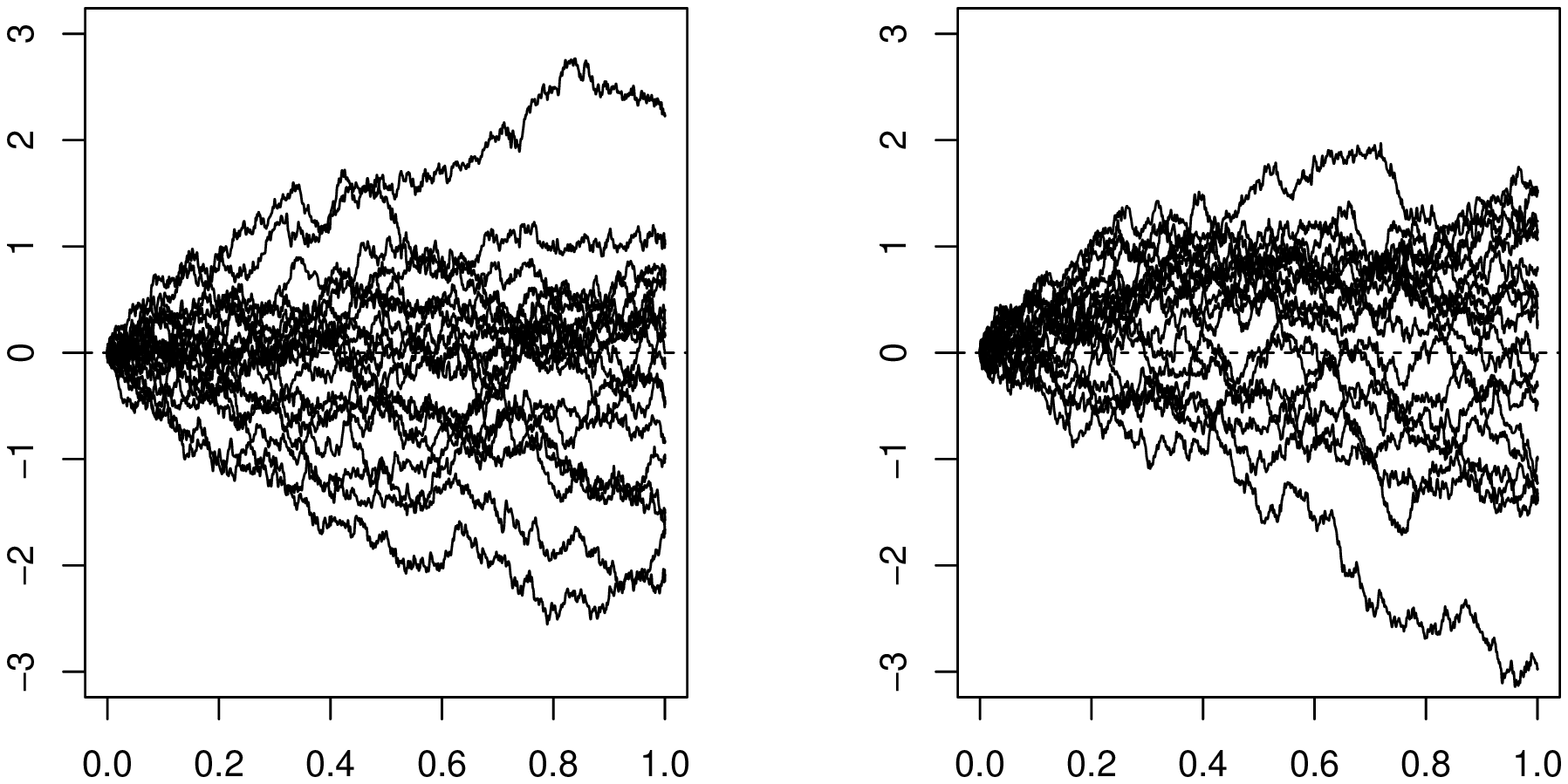}
\caption{Left panel: 20 realizations of the Brownian motion;
Right panel: independent 20 realizations of the Brownian motion with
the curve $at(1-t), a =1.5$ added. }
\label{fig:a-1-5}
\end{figure}

Table~\ref{ES-CI} shows the
empirical sizes $\hat{p}$, i.e. the fraction of rejections,
as well as asymptotic $90\%$ confidence intervals
\begin{align}
\left[\hat{p}-1.654\sqrt{\frac{\hat{p}(1-\hat{p})}{R}}, \ \
\hat{p}+1.654\sqrt{\frac{\hat{p}(1-\hat{p})}{R}}\right].\label{CI-def}
\end{align}
for the probability $p$ of rejection. The entries are based on $R=1,000$
replications. The table shows that the test based on convergence
\eqref{e:wc} has correct empirical size at the 5\% level and is a bit
too conservative at the 10\% level. However even at the 10\% level
the empirical sizes for $d\ge 3$ are not significantly different; they all
fall into each others 90\% confidence intervals. This illustrates the
main point that for the tests that use the asymptotics with $d\to \infty$
developed in the paper, selecting $d$ is not essential; every sufficiently
large $d$ gives the same conclusion on the significance.

The empirical power of the test is reported in Table~\ref{power}.
Again, for $d\ge 3$, the power remains statistically the same.  We
note that the change in mean equal to the function $at(1-t)$ with
$a=1.5$ is fairly small if the ``noise curves'' are Brownian motions.
This is illustrated in Figure~\ref{fig:a-1-5} which shows 20 Brownian
motions in the left panel and another independent sample of 20
Brownian motions with the curve $at(1-t), a=1.5$ added. If one knows
that this curve was added, one can discern it in the plot in the right
panel, but the difference would have been much less obvious if
individual curves were observed, as in the change--point setting
relevant to Table~\ref{power}.

Regarding Corollary~\ref{col}, we found out that the test based on
convergence \eqref{col-2} has empirical size only slightly higher than
nominal (about 1\% at 5\% level).  For $d\ge 3$, the empirical size
does not depend on $d$.  The test based on \eqref{col-3} severely
overrejects for $N=100$, and we do not recommend it.  The test based
on Theorem~\ref{th-two-1} overrejects by about 2\% at the 5\% level,
and by about 1\% at the 10\% level.  The power of the test is above
95\% for $N, M=100$ and $a=1.0$, and practically 100\% for larger $a$
or $N,M$. For $d\ge 2$, the rejection probabilities do not depend on
$d$.

\bigskip

\noindent{\bf Change--point analysis of annual temperature profiles.}
The goal of
this section is to illustrate the application of the change--point test
based on convergence \eqref{e:wc}.  Change--point analysis is an
important field of statistics with a large number of applications, the
recent monographs of \citet{chen:gupta:2011} and
\citet{basseville:nikifirov:tartakovsky:2012} provide numerous
references. The change--point problem in the context of functional data
has also received some attention, we refer to \citet{HKbook}
for the references,  \citet{aston:kirch:2012} report some most
recent research.

The data set we study consists of 156 years (1856-2011) of minimum
daily temperatures in Melbourne. These data are available at
\url{www.bom.gov.au} (the Australian Bureau of Meteorology website).
The original data can be viewed as 156 curves with 365 measurements on
each curve.  We converted them to functional objects in \texttt{R}
using 49 Fourier basis functions. Five consecutive functions are shown
in Figure~\ref{fig:5curves}. It is important to emphasize the
difference between the data we use and the Canadian temperature data
made popular by the books of \citet{ramsay:silverman:2005} and
\citet{ramsay:hooker:graves:2009}. The Canadian temperature curves are
the curves at 35 locations in Canada obtained by averaging annual
temperature over forty years. Since each such curve is an average of
forty curves like those shown in Figure~\ref{fig:5curves}, those
curves are much smoother, and the first two FPC's are sufficient to
describe their variability. Even after smoothing with 49 Fourier
functions, the annual temperature curves exhibit noticeable year to
year variability, and a larger number of FPC's is needed to capture
it, see Table~\ref{tab:5}.  The goals of our analysis are also different
from those of \citet{ramsay:silverman:2005}. We are interested in
detecting a change in the mean function using a sequence of noisy
curves; the examples in \citet{ramsay:silverman:2005} used the
averaged curves to describe static regression type dependencies
between climatic variables.

\begin{figure}
 \centering
 \includegraphics[scale = 0.775,angle = 270]{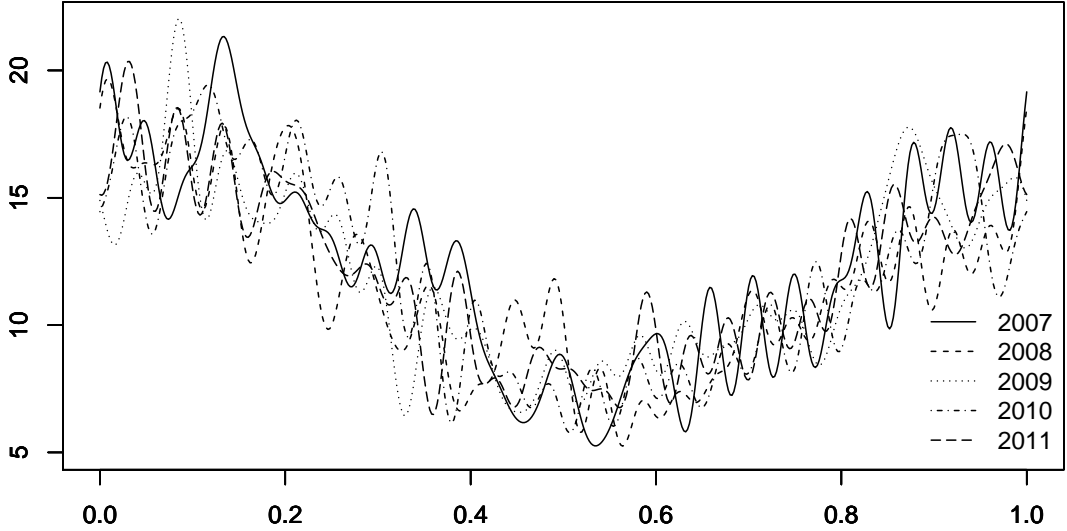}
\caption{Five annual temperature curves represented as functional objects. }
\label{fig:5curves}
\end{figure}

\begin{figure}
 \centering
 \includegraphics[scale = 0.775,angle = 270]{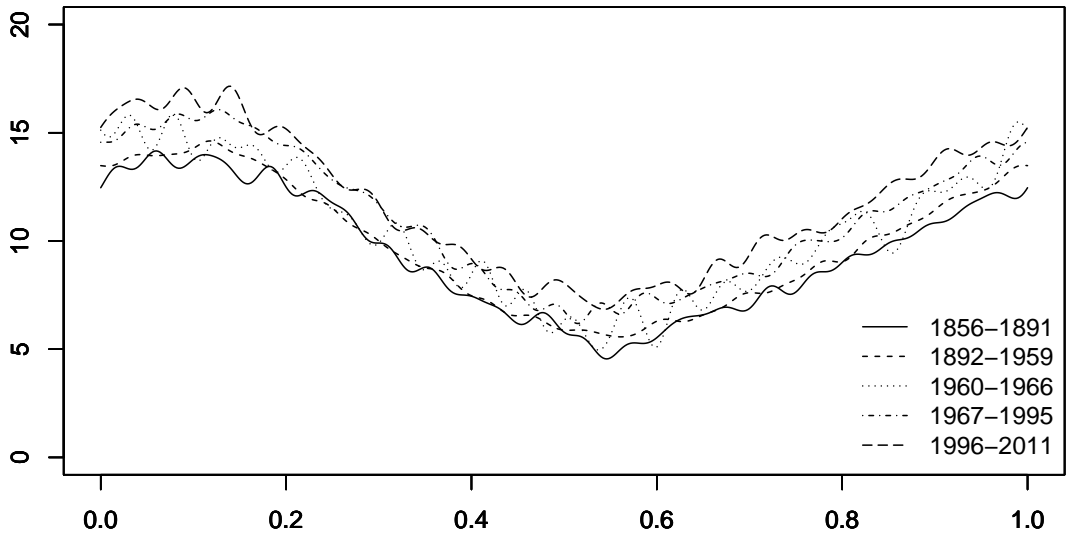}
\caption{Average temperature functions in the estimated partition segments.}
\label{fig:4}
\end{figure}

\begin{table}
{
\begin{center}
\begin{tabular}{ c c c c c c c c c } \toprule
\multicolumn{1}{ c }{$k$}&\multicolumn{1}{c }{1}&\multicolumn{1}{c }{2}&\multicolumn{1}{c }{3}&\multicolumn{1}{c }{4}&\multicolumn{1}{c }{5}&\multicolumn{1}{c }{6}&\multicolumn{1}{c }{7}&\multicolumn{1}{c }{8}\\
\cmidrule{1-1}\cmidrule(lr){2-9}
$\hat{\lambda}_k$&0.7151&0.1469&0.1295&0.1154&0.1046&0.1021&0.0944&0.0868\\
$f_p$&0.2248&0.2711&0.3118&0.3480&0.3809&0.4130&0.4427&0.4700\\
\midrule
\multicolumn{1}{ c }{$k$}&\multicolumn{1}{c }{9}&\multicolumn{1}{c }{10}&\multicolumn{1}{c }{11}&\multicolumn{1}{c }{12}&\multicolumn{1}{c }{13}&\multicolumn{1}{c }{14}&\multicolumn{1}{c }{15}&\multicolumn{1}{c }{16}\\
\cmidrule{1-1}\cmidrule(lr){2-9}
$\hat{\lambda}_k$&0.0845&0.0833&0.0758&0.0732&0.0726&0.0687&0.0661&0.0641\\
$f_p$&0.4966&0.5228&0.5466&0.5696&0.5925&0.6141&0.6349&0.6550\\
\midrule
\multicolumn{1}{ c }{$k$}&\multicolumn{1}{c }{17}&\multicolumn{1}{c }{18}&\multicolumn{1}{c }{19}&\multicolumn{1}{c }{20}&\multicolumn{1}{c }{21}&\multicolumn{1}{c }{22}&\multicolumn{1}{c }{23}&\multicolumn{1}{c }{24}\\
\cmidrule{1-1}\cmidrule(lr){2-9}
$\hat{\lambda}_k$&0.0620&0.0586&0.0559&0.0559&0.0534&0.0508&0.0472&0.0463\\
$f_p$&0.6745&0.6930&0.7105&0.7281&0.7449&0.7609&0.7757&0.7903\\
\midrule
\multicolumn{1}{ c }{$k$}&\multicolumn{1}{c }{25}&\multicolumn{1}{c }{26}&\multicolumn{1}{c }{27}&\multicolumn{1}{c }{28}&\multicolumn{1}{c }{29}&\multicolumn{1}{c }{30}&\multicolumn{1}{c }{31}&\multicolumn{1}{c }{32}\\
\cmidrule{1-1}\cmidrule(lr){2-9}
$\hat{\lambda}_k$&0.0440&0.0427&0.0426&0.0400&0.0377&0.0367&0.0359&0.0325\\
$f_p$&0.8041&0.8175&0.8309&0.8435&0.8553&0.8669&0.8782&0.8884\\
\midrule
\multicolumn{1}{ l  }{$k$}&\multicolumn{1}{c }{33}&\multicolumn{1}{c }{34}&\multicolumn{1}{c }{35}&\multicolumn{1}{c }{36}&\multicolumn{1}{c }{37}&\multicolumn{1}{c }{38}&\multicolumn{1}{c }{39}&\multicolumn{1}{c }{40}\\
\cmidrule{1-1}\cmidrule(lr){2-9}
$\hat{\lambda}_k$&0.0320&0.0299&0.0281&0.0274&0.0252&0.0248&0.0228&0.0211\\
$f_p$&0.8985&0.9079&0.9167&0.9253&0.9332&0.9410&0.9482&0.9548\\
\midrule
\multicolumn{1}{ l  }{$k$}&\multicolumn{1}{c }{41}&\multicolumn{1}{c }{42}&\multicolumn{1}{c }{43}&\multicolumn{1}{c }{44}&\multicolumn{1}{c }{45}&\multicolumn{1}{c }{46}&\multicolumn{1}{c }{47}&\multicolumn{1}{c }{48}\\
\cmidrule{1-1}\cmidrule(lr){2-9}
$\hat{\lambda}_k$&0.0207&0.0201&0.0188&0.0171&0.0166&0.0163&0.0129&0.0114\\
$f_p$&0.9614&0.9677&0.9736&0.9790&0.9842&0.9893&0.9934&0.9969\\
\bottomrule
\end{tabular}
\end{center}
}

\medskip

\caption{Eigenvalues and percentage of variance explained by the first $k$
eigenvalues, i.e.
$f_k = \sum_{i = 1}^k\hat{\lambda}_i/\sum_{j = 1}^N\hat{\lambda}_j$,
for $k = 1,2,\ldots,49$.}
\label{tab:5}
\end{table}

The analysis proceeds through the usual binary segmentation procedure.
The test is first applied to the whole data set. If the P--value is small,
the change--point is estimated as
\[
 \hat{\theta}_N = \inf\{k: I_N(k) = \sup_{1\leq j\leq N} I_N(j)\},
\]
where
\[
 I_N(\ell) = \frac {1}{d^2}\sum_{i = 1}^{d-1}\left(\sum_{j = 1}^i\left\{\frac 1N\left[\hat{S}_j(\ell ) - \frac \ell N \hat{S}_j(N)\right]^2 - \frac \ell N\left(\frac{N-\ell }{N}\right)\right\}\right)^2.
\]
($I_N$ is a discretization of $\hat Z_N$.) The test is then applied to
the two segments, and the procedure continues until no change--points
are detected. In practice, a procedure of this type detects only a few
change--points (four in our case), so the problems of multiple testing
are not an issue.  We applied the test using many values of $d$, and
we were pleased to see that the final segmentation does not depend on
$d$.  Table~\ref{tab:4} shows the outcome. The estimated change--points
are the years 1892, 1960, 1967, 1996.  It is clear that the change--point
model is not an exact climatological model for the evolution of
annual temperature curves, but it is popular in climate studies, see
e.g. \citet{gallagher:lund:robbins:2012}, as it allows us  to attach
statistical significance to conclusions and provides periods of
approximately constant mean temperature profiles. In this light, the
weak evidence for a change--point in 1967 could be viewed as
indicating an accelerated change in the period 1960--1995.  The
estimated mean temperature curves over the segments of approximately
constant mean are shown in Figure~\ref{fig:4}. An increasing pattern
of the mean temperature is seen; the mean curve shifted upwards by
about two degrees Celsius over the last 150 years. This could be due
to the conjectured global temperature increase or the urbanization of
the Melbourne area, or a combination of both.  A discussion of such
issues is however beyond the intended scope of this paper.

\begin{table}
{
\begin{center}
\begin{tabular}{lccrrrr} \toprule
It.& Segment & Estimated & \multicolumn{4}{c}{P-value}\\
\cmidrule(lr){4-7}
&&change--point&\multicolumn{1}{c}{$d = 3$}&\multicolumn{1}{c}{$d = 4$}&\multicolumn{1}{c}{$d = 5$}&\multicolumn{1}{c}{$d = 6$}\\
\cmidrule(lr){4-4}\cmidrule(lr){5-5}\cmidrule(lr){6-6}\cmidrule(lr){7-7}
1&1856-2011&1960&~0.0000&~0.0000&~0.0000&~0.0000\\
2&1856-1959&1892&0.0000~&0.0000~&0.0000~&0.0000~\\
3&1856-1891&---&0.1865~&0.2323~&0.3524~&0.4822~\\
4&1892-1959&---&0.9522~&0.9690~&0.9256~&0.6561~\\
5&1960-2011&1996&0.0000~&0.0000~&0.0000~&0.0000~\\
6&1960-1995&1967&0.0013~&0.0011~&0.0025~&0.0017~\\
7&1960-1966&---&0.9568~&0.9549~&0.9818~&0.9935~\\
8&1967-1995&---&0.2927~&0.4305~&0.1786~&0.1348~\\
9&1996-2011&---&0.4285~&0.5345~&0.6413~&0.7365~\\\midrule
It.& Segment & Estimated& \multicolumn{4}{c}{P-value}\\
\cmidrule(lr){4-7}
&&change--point&\multicolumn{1}{c}{$d = 7$}&\multicolumn{1}{c}{$d = 8$}&\multicolumn{1}{c}{$d = 9$}&\multicolumn{1}{c}{$d = 10$}\\
\cmidrule(lr){4-4}\cmidrule(lr){5-5}\cmidrule(lr){6-6}\cmidrule(lr){7-7}
1&1856-2011&1960&~0.0000&~0.0000&~0.0000&~0.0000\\
2&1856-1959&1892&0.0000~&0.0000~&0.0000~&0.0000~\\
3&1856-1891&---&0.4235~&0.4325~&0.4901~&0.5667~\\
4&1892-1959&---&0.4646~&0.4348~&0.4696~&0.5068~\\
5&1960-2011&1996&0.0000~&0.0000~&0.0000~&0.0000~\\
6&1960-1995&1967&0.0026~&0.0038~&0.0058~&0.0067~\\
7&1960-1966&---&0.9992~&\multicolumn{1}{c}{---}&\multicolumn{1}{c}{---}&\multicolumn{1}{c}{---}\\
8&1967-1995&---&0.1245~&0.0690~&0.0571~&0.0586~\\
9&1996-2011&---&0.8243~&0.9118~&0.9618~&0.9779~\\
\bottomrule
\end{tabular}
\end{center}
}
\medskip

\caption{Segmentation procedure of the data into periods
with constant mean function}
\label{tab:4}
\end{table}

\clearpage

\appendix

\section{\bf Proof of Theorem~\ref{approx}
} \label{s:p-approx}

We start with some elementary properties of the projections
$\xi_{i,j}$. Let $|\cdot |$ denote the Euclidean norm of vectors.
\begin{lemma}\label{l-2.1} If Assumptions \ref{as-1}, \ref{as-5}
and \ref{as-6} hold, then
\beq\label{eq-2.1}
E\bxi_1={\bf 0},
\eeq
\beq\label{eq-2.2}
E\bxi_1\bxi_1^T={\bf I}_d,
\eeq
where ${\bf I}_d$ is the $d\times d$ identity matrix. Moreover,
\beq\label{eq-2.3}
E|\bxi_1|^3\leq E\norm Z_1\norm^3
\left(\sum_{j=1}^d 1/\lambda_j\right)^{3/2}
\eeq
and for all  $1\leq j \leq d$
\beq\label{eq-2.3a}
E|\xi_{1,j}|^3\leq E\norm Z_1\norm^3 /\lambda_j^{3/2}.
\eeq
\end{lemma}
\begin{proof}
Since $EZ_1(t)=0$, the relation in (\ref{eq-2.1}) is obvious.
 The orthonormal functions $v_k$ and $v_\ell$ satisfy (\ref{eq-eig}),
so we get
\begin{equation*}
E\xi_{i,k}\xi_{i,\ell}=\frac{1}{(\lambda_k\lambda_\ell)^{1/2}}
\intt {\mathfrak c}(t,s)v_k(s)v_\ell(s)dtds=
\left\{
\begin{array}{ll}
0,\;\;&\mbox{if}\;\;k\neq \ell
\vspace{.3 cm}\\
1,\;\;&\mbox{if}\;\;k= \ell,
\end{array}
\right.
\end{equation*}
proving (\ref{eq-2.2}).  Using  the definition of the Euclidean norm
and the Cauchy--Schwarz inequality we conclude
$$
|\bxi_1|^3=\left(\sum_{j=1}^d\la Z_1, v_j\ra^2/\lambda_j\right)^{3/2}\leq \left(\sum_{j=1}^d
\norm Z_1\norm^2\norm v_j\norm^2/\lambda_j\right)^{3/2}=\norm Z_1\norm^3\left(\sum_{j=1}^d1/\lambda_j\right)^{3/2},
$$
since $\norm v_j\norm=1$. Taking the expected value of the equation above we obtain (\ref{eq-2.3}). Clearly,
$$
E|\xi_{1,j}|^3=\lambda_j^{-3/2}E|\la Z_1, v_j\ra|^3\leq\lambda_j^{-3/2}E\norm Z_1\norm^3.
$$
\end{proof}

The next lemma plays a central role in the proof of Theorem~\ref{approx}.

\begin{lemma}\label{sena} If Assumptions \ref{as-1}, \ref{as-5} and \ref{as-6}  hold, then for all $n$ we can define  independent identically distributed standard normal vectors $\bgamma_1, \ldots, \bgamma_n$ in $R^d$ such that
$$
P\left\{ \left|\sum_{i=1}^n\bxi_i -\sum_{i=1}^n\bgamma_i\right|\geq cn^{3/8}d^{1/4}(E|\bxi_1|^3+E|\bgamma_1|^3)^{1/4}\right\}
\leq cn^{-1/8}d^{1/4}(E|\bxi_1|^3+E|\bgamma_1|^3)^{1/4},
$$
where $c$ is an absolute constant.
\end{lemma}
\begin{proof}
The result is a consequence of Theorem 6.4.1 on p.\ 207 of
\citet{senatov:1998} and the corollary to Theorem 11
 in \citet{strassen:1965}.
\end{proof}
We note that
\beq\label{eq-2.4}
(E|\bxi_1|^3+E|\bgamma_1|^3)^{1/4}\leq  (E|\bxi_1|^3)^{1/4}+(E|\bgamma_1|^3)^{1/4}.
\eeq
Also, since $|\bgamma_1|^2$ is the sum of the squares of $d$
independent standard normal random variables, Minkowski's
inequality implies
\beq\label{eq-2.5}
E|\bgamma_1|^3\leq c_1d^{3/2},
\eeq
with some constant $c_1$, and clearly
\beq\label{eq-2.6}
d^{3/2}\leq \lambda_1^{3/2}\left(\sum_{\ell=1}^d1/\lambda_\ell\right)^{3/2}.
\eeq
Combining Lemma \ref{sena} with (\ref{eq-2.4})--(\ref{eq-2.6}),
 we conclude that
\begin{align}
\label{eq-2.7}
P\Biggl\{\left|\sum_{i=1}^n\bxi_i -\sum_{i=1}^n\bgamma_i\right|\geq c_2n^{3/8}d^{1/4}\left(\sum_{j=1}^d1/\lambda_j\right)^{3/8}\Biggr\}\leq c_2n^{-1/8}d^{1/4}\left(\sum_{j=1}^d1/\lambda_j\right)^{3/8},
\end{align}
where $c_2$ does not depend on $d$.\\
In the next lemma we provide an upper bound for the variance of
$\sum_i^n(\xi_{i,j}-\gamma_{i,j})$, where
$\bgamma_i=(\gamma_{i,1},\ldots, \gamma_{i,d})^T$
is defined in Lemma \ref{sena}.
\begin{lemma}\label{sena-2} If Assumptions \ref{as-1}, \ref{as-5}
and \ref{as-6}  hold, then
for any  $1\leq j \leq d$ we get
$$
E\left(\sum_{i=1}^n\xi_{i,j}-\sum_{i=1}^n\gamma_{i,j}\right)^2\leq c_3n^{23/24}\frac{1}{\lambda_j}\left(d^{1/4}\left(\sum_{\ell=1}^d 1/\lambda_\ell\right)^{3/8}\right)^{1/3},
$$
where $c_3$ does not depend on $d$.
\end{lemma}
\begin{proof} Let
$$
U_n(j)=n^{-1/2}\sum_{i=1}^n(\xi_{i,j}-\gamma_{i,j})\;\;\;\mbox{and}\;\;\; r_n=c_2n^{-1/8}d^{1/4}\left( \sum_{\ell=1}^d 1/\lambda_\ell\right)^{3/8}.
$$
First we write
\begin{align*}
EU_n^2(j)&=E[U_n^2(j)I\{|U_n(j)|\leq r_n\}]+E[U_n^2(j)I\{|U_n(j)|> r_n\}]\\
&\leq r_n^2+\frac{2}{n}E\biggl[\biggl(\sum_{i=1}^n\xi_{i,j}\biggl)^2I\{|U_n(j)|>r_n\}\biggl]+\frac{2}{n}
E\biggl[\biggl(\sum_{i=1}^n
\gamma_{i,j}\biggl)^2I\{|U_n(j)|>r_n\}\biggl].
\end{align*}
Using H\"older's inequality we get that
\begin{align*}
E\biggl[\biggl(\sum_{i=1}^n\xi_{i,j}\biggl)^2I\{|U_n(j)|>r_n\}\biggl]&\leq E\biggl[\biggl|\sum_{i=1}^n\xi_{i,j}\biggl|^3\biggl]^{2/3}\biggl[  P\{|U_n(j)|>r_n\}\biggl]^{1/3}\\
&\leq E\biggl[\biggl|\sum_{i=1}^n\xi_{i,j}\biggl|^3\biggl]^{2/3}r_n^{1/3}
\end{align*}
by (\ref{eq-2.7}). Applying now Rosenthal's inequality
(cf.\ \citet{petrov:1995}, p.\ 59) we obtain
$$
E\biggl|\sum_{i=1}^n\xi_{i,j}\biggl|^3
\leq c_4\biggl\{ \sum_{i=1}^nE|\xi_{i,j}|^3 +
\biggl(\sum_{i=1}^nE\xi_{i,j}^2\biggl)^{3/2}\biggl\},
$$
where $c_4$ is an absolute constant. Hence
$$
E\biggl|\sum_{i=1}^n\xi_{i,j}\biggl|^3
\leq c_5\{n\lambda_j^{-3/2}+n^{3/2}\}\leq c_6(n/\lambda_j)^{3/2}
$$
and therefore
\begin{align*}
E\biggl[\biggl(\sum_{i=1}^n\xi_{i,j}\biggl)^2I\{|U_n(j)|>r_n\}\biggl]
&\leq c_7(n/\lambda_j)r_n^{1/3}\\
&\leq
c_8 n^{23/24}\frac{1}{\lambda_j}\biggl(d^{1/4}
\biggl(\sum_{\ell=1}^d1/\lambda_\ell\biggl)^{3/8}\biggl)^{1/3}.
\end{align*}
Following the previous arguments one can show that
$$
E\biggl[\biggl(\sum_{i=1}^n\gamma_{i,j}\biggl)^2I\{|U_n(j)|>r_n\}\biggl]\leq c_9 n^{23/24}\frac{1}{\lambda_j}\biggl(d^{1/4}
\biggl(\sum_{\ell=1}^d1/\lambda_\ell\biggl)^{3/8}\biggl)^{1/3}.
$$
The constants $c_8$ and $c_9$ do not depend on $d$. Since in view of Assumption \ref{d-2}, $nr_n^2$ is smaller than the latter rates,
this completes the proof of Lemma \ref{sena-2}.
\end{proof}

\begin{proof}[\bf Proof of Theorem~\ref{approx}]
  We use a blocking argument to construct a Wiener process which is
  close to the partial sums $\sum_{1\leq i \leq k}\xi_{i,j}, 1\leq k
  \leq N, 1\leq j \leq d$. Let $K$ be the length of the blocks to be
  chosen later. Let $M=\lfloor N/K\rfloor$. For $k=\ell M, 1\leq \ell
  \leq K$ we write
\[
\sum_{i=1}^k\xi_{i,j}=\sum_{v=1}^\ell
\biggl(\sum_{i=(v-1)M+1}^{vM}\xi_{i,j}\biggl).
\]
Using the $\gamma_{i,j}$'s, the independent standard normal random
variables constructed in Lemma \ref{sena},  we define
\beq\label{wien}
W_j(k)=\sum_{i=1}^k\gamma_{i,j},\;\;\;1\leq j \leq d,\; 1\leq k \leq N.
\eeq
By Lemma \ref{sena-2} we get for any
$0<\delta <1/2$ and $1\leq j \leq d$ via Kolmogorov's inequality
 (cf.\ \citet{petrov:1995}), p.\ 54)
\begin{align}\label{kol}
P\biggl\{&\max_{1\leq \ell \leq K}\biggl|
\sum_{i=1}^{\ell M}\xi_{i,j}-W_j(\ell M)\biggl|
\geq N^{1/2-\delta}\biggl\}\\
&= P\biggl\{\max_{1\leq \ell \leq K}\biggl|\sum_{v=1}^{\ell }\biggl(\sum_{i=(v-1)M+1}^{vM}(\xi_{i,j}-\gamma_{i,j})\biggl)\biggl|\geq N^{1/2-\delta}\biggl\} \notag \\
\notag
&\leq \frac{1}{N^{1-2\delta}}\sum_{v=1}^K E\biggl(\sum_{i=(v-1)M+1}^{vM}(\xi_{i,j}-\gamma_{i,j})\biggl)^2\\
&\leq \frac{c_3}{N^{1-2\delta}}KM^{23/24}  \frac{1}{\lambda_j} \biggl(d^{1/4}\biggl(\sum_{\ell=1}^d1/\lambda_\ell\biggl)^{3/8}\biggl)^{1/3}                                                              \notag\\
&\leq c_3 N^{2\delta-1/24}K^{1/24}
\frac{1}{\lambda_j} \biggl(d^{1/4}
\biggl(\sum_{\ell=1}^d1/\lambda_\ell\biggl)^{3/8}\biggl)^{1/3}.
\notag
\end{align}
One can define independent Wiener processes (standard Brownian
motions) $W_j(x), x\geq 0, 1\leq j \leq d$ such that (\ref{wien})
holds.
We obtained approximations for the partial sums of the $\xi_{i,j}$'s at the points $k=\ell M, 1\leq \ell\leq  K.$ Next we show that neither the partial sums of the $\xi_{i,j}$'s nor the Wiener processes $W_j(x)$ can oscillate too much between $\ell M$ and $(\ell +1)M$. \\
Using again Rosenthal's inequality (cf.\ \citet{petrov:1995}, p.\ 59)
we obtain for all $1\leq j \leq d$ that
\begin{align}\label{rose}
E\biggl|\sum_{i=1}^M\xi_{i,j}\biggl|^3
&\leq c_{10}\biggl\{ \sum_{i=1}^ME|\xi_{i,j}|^3 +
\biggl(\sum_{i=1}^ME\xi_{i,j}^2\biggl)^{3/2}\biggl\}\\
&\leq c_{11}\{M/\lambda_j^{3/2} +M^{3/2}\}    \notag\\
&\leq  c_{11} (1+\lambda_1^{3/2})(M/\lambda_j)^{3/2}\notag
\end{align}
on account of Lemma \ref{l-2.1}.
Combining the Marcinkiewicz--Zygmund inequality
 (cf.\ \cite{petrov:1995}, p.\ 82) with (\ref{rose}) we conclude
\beq\label{mz}
E\biggl(\max_{1\leq h \leq M}\biggl|
\sum_{i=1}^h\xi_{i,j}\biggl|\biggl)^3\leq c_{12}(M/\lambda_j)^{3/2}.
\eeq
Applying (\ref{mz}) we get
\begin{align}\label{max-1}
P\biggl\{&\max_{0\leq \ell \leq K+1}\max_{1\leq h \leq M}\biggl|\sum_{i=1}^{\ell M}\xi_{i,j}
-\sum_{i=1}^{\ell M+h}\xi_{i,j}\biggl|\geq N^{1/2-\delta}\biggl\}\\
&\leq (K+2) P\biggl\{\max_{1\leq h \leq M}\biggl|\sum_{i=1}^h\xi_{i,j}\biggl|>N^{1/2-\delta}\biggl\}
\notag\\
&\leq \frac{c_{13}}{N^{3/2-3\delta}}K(M/\lambda_j)^{3/2}\notag\\
&\leq c_{13}N^{3\delta}K^{-1/2}\lambda_j^{-3/2}. \notag
\end{align}
Lemma 1.2.1 of \cite{csorgo:revesz:1981} yields
\begin{align}\label{max-2}P\biggl\{\max_{0\leq \ell \leq K}
\sup_{|h|\leq M}|W_{j}(\ell M)-W_j(\ell M+h)|
\geq c_{14} M^{1/2}(\log N)^{1/2}\biggl\}\leq \frac{c_{15}}{N^2}.
\end{align}
Now choosing  $\delta =1/80$ and $K=\lfloor N^\beta \rfloor$ with $\beta =1/10$, it follows from (\ref{kol}), (\ref{max-1}) and (\ref{max-2}) for all $1\leq j \leq d$ that
\begin{align}\label{last}
P\biggl\{&\sup_{0\leq y \leq N}\biggl|\sum_{1\leq i \leq y}\xi_{i,j}-W_j(y)\biggl|>N^{1/2-\delta}\biggl\}\\
&\leq c_{15}N^{-\delta}\biggl\{\frac{1}{\lambda_j}\biggl(d^{1/4}\biggl(\sum_{\ell=1}^d1/\lambda_\ell\biggl)^{3/8}\biggl)^{1/3}+
\frac{1}{\lambda^{3/2}_j}\biggl\}.\notag
\end{align}
The result now follows from (\ref{last}) with $W_{j,N}(x)=N^{-1/2}W_j(Nx), 0\leq x \leq 1.$
\end{proof}

\section{\bf Proofs of the results of
Section~\ref{s:cp}} \label{s:p-cp}

We first  investigate the weak convergence of the process
\[
{Z}_N(u,x)=\frac{1}{d^{1/2}}\sum_{j=1}^{\lfloor du\rfloor}
\left\{ ({S}_{j,N}(x )-x{S}_{j,N}(1))^2-x(1-x)\right\},\;\;0\leq u, x \leq 1,
\]
with $S_{j,N}(x)$ given by \eqref{n-sum}.
The  difference between $\hat{Z}_N(u,x)$ and ${Z}_N(u,x)$ is that
$\hat{Z}_N$ is computed from the empirical projections $\hat{v}_1,
\ldots, \hat{v}_d$, while $Z_N$ is based on the unknown population
eigenfunctions $v_1, \ldots, v_d$.

\begin{theorem}\label{pure} If Assumptions \ref{as-1}, \ref{as-5},
\ref{as-6} and \ref{d-0}--\ref{d-3} hold, then
$$
{Z}_N(u,x)\;\;\to\;\;\Gamma(u,x)\;\;\mbox{in}\;\;{\mathcal D}[0,1]^2,
$$
where the Gaussian process $\Gamma(u,x)$ is defined in Theorem \ref{th-1}.
\end{theorem}

To  prove  Theorem \ref{pure}, we need several lemmas and some
additional notation.

Let
$$
V_{j,N}(x)=S_{j,N}(x)-xS_{j,N}(1)\;\;\;\mbox{and}\;\;\;
B_{j,N}(x)=W_{j,N}(x)-xW_{j,N}(1),
$$
where $S_{j,N}$ is defined in (\ref{n-sum}) and the $W_{j,N}$'s are
the Wiener processes of Theorem \ref{approx}.
It follows from the definition that for each $N$ the processes
$B_{j,N}, 1\leq j \leq d,$ are independent Brownian bridges.
\begin{lemma}\label{bridge-approx} If Assumptions \ref{as-1},
\ref{as-5} and \ref{as-6}  hold, then
\begin{align*}
P\biggl\{\sup_{0\leq x \leq 1}\sum_{j=1}^d &\bigl|V^2_{j,N}(x)-B_{j,N}^2(x)\bigl|\geq    20dN^{-1/80}(\log N)^{1/2}\biggl\}\\
&\leq c_*N^{-1/80}\biggl\{d^{1/12}\biggl(\sum_{\ell=1}^d1/\lambda_\ell\biggl)^{1/8}+\sum_{j=1}^d1/\lambda^{3/2}_j
\biggl\}+c_{**}dN^{-2},
\end{align*}
where $c_*$ and $c_{**}$ only depend on $\lambda_1$ and $E\norm Z_1\norm^3.$
\end{lemma}
\begin{proof} First we write
$$
V^2_{j,N}(x)-B_{j,N}^2(x)=(V_{j,N}(x)-B_{j,N}(x))^2
+2B_{j,N}(x)(V_{j,N}(x)-B_{j,N}(x)).
$$
Since the $B_{j,N}$'s are Brownian bridges, the distribution of the
supremum functional of the Brownian bridge
 (cf.\ \cite{csorgo:revesz:1981}) gives
\[
 P\biggl\{\max_{1\leq j \leq d}\sup_{0\leq
x \leq 1}|B_{j,N}(x)|\geq 4(\log N)^{1/2}\biggl\}\leq
c_{**}\frac{d}{N^2},
\]
where $c_{**}$ is an absolute constant.  Now
the result follows immediately from Theorem \ref{approx}.
\end{proof}

Now we prove the weak convergence of the partial sums of the squares
of independent Brownian bridges. Let $B_1, B_2, \ldots, B_d$ be
independent Brownian bridges.
\begin{lemma}\label{hahn} As $d\to \infty$, we have that
$$
\frac{1}{d^{1/2}}\sum_{j=1}^{\lfloor du\rfloor}
\bigl(B_j^2(x)-x(1-x)\bigl)\;\to \;\Gamma(u,x)
\quad\mbox{in}\;\;{\mathcal D}[0,1]^2,
$$
where the Gaussian process $\Gamma(u,x)$ is defined in Theorem \ref{th-1}.
\end{lemma}
\begin{proof} The proof is based on Theorem 2 of \cite{hahn:1978}.
Let $B$ denote a Brownian  bridge and
$\theta_1=\sup_{0\leq t \leq 1}|B(t)|$. It is clear that
$E\theta_1^m<\infty$ for all $m\geq 1$.
According to \cite{garsia:1970},
 there is a random variable $\theta_2$ such that $E\theta_2^m<\infty$
for all $m\geq 1$ and
$$
|B(t)-B(s)|\leq \theta_2 (|t-s|\log(1/|t-s|))^{1/2},\;\;\;0\leq t,s\leq 1.
$$
Let $V(t)=B^2(t)-t(1-t)$. We note
$$
|V(t)-V(s)|\leq 2\theta_1\theta_2(|t-s|\log(1/|t-s|))^{1/2}+|t-s|.
$$
Thus we get
\beq\label{c-1}
E(V(t)-V(s))^2\leq c_{16}|t-s|\log(1/|t-s|)\;\;\;\mbox{for all}\;\;0\leq t,s\leq 1
\eeq
and
\begin{align}\label{c-2}
E[(V(t)-V(z))^2(V(z)-V(s))^2]
\leq c_{17}(|t-s|\log(1/|t-s|))^2
\end{align}
for all $0\leq s\leq z\leq t\leq 1.$
The estimates in (\ref{c-1}) and (\ref{c-2}) yield that the
conditions of Theorem 2 of \cite{hahn:1978} are satisfied,
 completing the proof Lemma \ref{hahn}.
\end{proof}
\medskip

\begin{proof}[\bf Proof of Theorem \ref{pure}] It follows immediately
from Lemmas \ref{bridge-approx} and \ref{hahn}.
\end{proof}

The transition from Theorem~\ref{pure} to Theorem~\ref{th-1} is based
on the following lemma, in which the norm is the Hilbert--Schmidt norm.

\begin{lemma}\label{dunsch} If Assumptions \ref{as-1}, \ref{as-2}
and  \ref{as-5} hold, then
\beq\label{lamb}
|\lambda_j-\hat{\lambda}_j|
\leq \norm {\mathfrak c}-\hat{\mathfrak c}\norm
\eeq
and
\beq\label{vee}
\norm v_j-\hat{c}_j\hat{v}_j\norm\leq
\frac{2\sqrt{2}}{\zeta_j} \norm {\mathfrak c}-\hat{\mathfrak c}\norm,
\eeq
where $\hat{c}_j= {\rm sign}(\langle \hat{v}_j, v_j \rangle)$
are random signs,
and  $\zeta_1, \zeta_2, \ldots $ are defined in Assumption \ref{d-4}.
\end{lemma}
\begin{proof} Inequality (\ref{lamb}) can be deduced
from the general results presented  in Section VI.1 of
\cite{gohberg:1990}
or in \cite{dunford:schwartz:1988}. These results are presented
in a convenient form in Lemma 2.2 in \cite{HKbook}.
Finally Lemma 2.3 in \cite{HKbook} gives (\ref{vee}).
\end{proof}

\begin{proof}[\bf Proof of Theorem \ref{th-1}]
Introducing
$$
U_N(x)=U_N(x,t)=\frac{1}{N^{1/2}}
\biggl\{\sum_{i=1}^{\lfloor Nx\rfloor}Z_i(t)-x\sum_{i=1}^N Z_i(t)\biggl\}
$$
we can write
$$
\hat{Z}_N(u,x)=\frac{1}{d^{1/2}}\sum_{j=1}^{\lfloor du\rfloor}
\biggl\{\frac{1}{\hat{\lambda}_j}\la U_N(x),\hat{v}_j\ra^2-x(1-x)\biggl\}.
$$
Elementary arguments give
\begin{align*}
\sum_{j=1}^{\lfloor du\rfloor}\frac{1}{\hat{\lambda}_j}\la U_N(x),\hat{v}_j\ra^2
=\sum_{j=1}^{\lfloor du\rfloor}\frac{1}{{\lambda}_j}&\la U_N(x),\hat{c}_j{v}_j\ra^2
+\sum_{j=1}^{\lfloor du\rfloor}\biggl\{\frac{1}{\hat{\lambda}_j}-\frac{1}{{\lambda}_j}\biggl\}\la U_N(x),\hat{v}_j\ra^2\\
&+\sum_{j=1}^{\lfloor du\rfloor}\frac{1}{{\lambda}_j}(\la U_N(x),\hat{v}_j\ra^2-\la U_N(x), \hat{c}_j{v}_j\ra^2).
\end{align*}
By the Cauchy--Schwarz inequality we have
\beq\label{first}
\frac{1}{d^{1/2}}\sum_{j=1}^{ d}\biggl|\frac{1}{\hat{\lambda}_j}-\frac{1}{{\lambda}_j}\biggl|\la U_N(x),\hat{v}_j\ra^2\leq \norm U_N(x)\norm^2\frac{1}{d^{1/2}}\sum_{j=1}^{ d}\frac{|\lambda_j-\hat{\lambda}_j|}{\hat{\lambda_j}\lambda_j}
\eeq
and since $|a^2-b^2|=|a+b||a-b|$,
\beq\label{second}
\frac{1}{d^{1/2}}\sum_{j=1}^{ d}\frac{1}{{\lambda}_j}(\la U_N(x),\hat{v}_j\ra^2-\la U_N(x)-\hat{c}_j{v}_j\ra^2)
\leq \norm U_N(x)\norm^2\frac{2}{d^{1/2}}\sum_{j=1}^{ d}\frac{1}{\lambda_j}\norm \hat{v}_j-\hat{c}_j{v}_j\norm^2.
\eeq
It follows from the results of \citet{kuelbs:1973}
(for a shorter proof we refer to Theorem 6.3 in
\citet{HKbook})
that
\[
\sup_{0\leq x \leq 1}\norm U_N(x)\norm^2=O_P(1).
\]
Due to Assumption \ref{as-6} we can use a
Marcinkiewicz--Zygmund type law of
large numbers for sums of independent and identically
distributed random functions in
Banach spaces
(cf., e.g., \citet{woyczynski:1978} or \citet{howell:taylor:1980})
to conclude
\[
\norm {\mathfrak c} - \hat{\mathfrak c} \norm=O_P(N^{-1/3}).
\]
Assumption \ref{d-3} gives that $N^{-1/120}/\lambda_d\to 0$
and therefore  by Lemma \ref{dunsch}
$$
\max_{1\leq i \leq d}\frac{\lambda_i}{\hat{\lambda}_i}=O_P(1).
$$
So by Lemma \ref{dunsch} and (\ref{first}) we have
\begin{align*}
\frac{1}{d^{1/2}}\sum_{j=1}^{ d}\biggl|\frac{1}{\hat{\lambda}_j}-\frac{1}{{\lambda}_j}\biggl|\la U_N(x),\hat{v}_j\ra^2&=O_P(1)\frac{1}{d^{1/2}N^{1/3}}\sum_{i=1}^d1/\lambda_i^2
\\
&=O_P(1)\frac{d^{1/2}}{N^{1/3}}\frac{1}{\lambda_d^2}
\\
&=O_P(1)\frac{N^{1/80}}{N^{1/3}}N^{1/60}
\\
&=o_P(1)
\end{align*}
on account of Assumptions \ref{d-1} and \ref{d-3}.
Similarly, (\ref{second}) and Assumption \ref{d-4} yield
\begin{align} \label{final}
\frac{1}{d^{1/2}}\sum_{j=1}^{ d}\frac{1}{{\lambda}_j}
\la U_N(x),\hat{v}_j-\hat{c}_j{v}_j\ra^2=O_P(1)\frac{1}{d^{1/2}N^{1/3}}\sum_{j=1}^d\frac{1}{\lambda_j\zeta_j}
=o_P(1).
\end{align}
Theorem \ref{th-1} now follows from Theorem \ref{pure}.
\end{proof}

\begin{proof}[\bf  Proof of Corollary \ref{col}]
By Lemma~\ref{bridge-approx} and \eqref{final},
relation \eqref{col-1} is proven if we show that
\beq\label{cp-1}
\frac{1}{d^{1/2}\sigma_0}\left\{\sum_{i=1}^d
\sup_{0\leq x \leq 1}B^2_i(x)-d\kappa_0\right\}\;\;
\stackrel{{\mathcal D}}{\to}\;\;N(0,1),
\eeq
where $B_1, B_2, \ldots, B_d$ are independent Brownian bridges.
Clearly, (\ref{cp-1}) is an immediate consequence of the central
limit theorem.
Similarly, to establish \eqref{col-2}, we need to show only that
$$
\frac{1}{(d/45)^{1/2}}\left\{\sum_{i=1}^d\int B^2_i(x)dx
-\frac{d}{6}\right\}\;\;\stackrel{{\mathcal D}}{\to}\;\;N(0,1).
$$
The above result is known, see  Remark 2.1 in
\citet{aue:hormann:horvath:reimherr:2009}.
The same argument can be used to prove \eqref{col-3}.
\end{proof}

\section{\bf Proofs of the results of
Section~\ref{s:two-s}} \label{s:p-two}
We note that under the null hypothesis
$\bar{X}_N-\bar{Y}_M=\bar{Z}_N-\bar{Q}_M.$
Define
$$
F_{N,M}=\sum_{j=1}^NZ_j-\frac{N}{M}\sum_{j=1}^MQ_j.
$$
The proof of Theorem \ref{th-two-1} is based on Lemma \ref{sena}, we need to write $F_{N,M}$ as a single sum of independent identically distributed random processes and an additional small remainder term. Let $K$ be an integer and define the integers $R=\lfloor N/K\rfloor$ and $L=\lfloor M/K\rfloor$. Next we define
$$
A_{i}=\sum_{\ell=R(i-1)+1}^{iR} Z_\ell-\sum_{\ell=L(i-1)+1}^{iL} \frac{N}{M}Q_\ell,\;\;\;i=1,2,\ldots ,K.
$$
Clearly,
$$
F_{N,M}=\sum_{i=1}^KA_i+\tilde{A},
$$
where
$$
\tilde{A}=\sum_{\ell=KR+1}^NZ_\ell-\frac{N}{M}\sum_{\ell=KL+1}^M Q_\ell.
$$

We will show first if $v$ is a function with $\norm v\norm =1$, then for every $n$
\beq\label{e-1}
E\left|\sum_{\ell=1}^n\la Z_\ell, v\ra\right|^3\leq c_1n^{3/2}
\eeq
and
\beq\label{e-2}
E\left|\sum_{\ell=1}^n\la Q_\ell, v\ra\right|^3\leq c_2n^{3/2},
\eeq
where $c_1$ and $c_2$ only depends on $E\norm Z_1\norm^3$  and $E\norm Q_1\norm^3$, respectively.
 Using Rosenthal's inequality (cf.\ \citet{petrov:1995}, p.\ 59) we get
$$
E\biggl|\sum_{\ell=1}^n\la Z_\ell, v\ra\biggl|^3\leq c_3\left\{ nE| \la Z_1, v\ra |^3+( nE\la Z_1, v\ra^2 )^{3/2} \right\},
$$
where $c_3$ is an absolute constant. It is easy to see that
$$
| \la Z_1, v\ra |\leq \norm Z_1\norm,
$$
which implies (\ref{e-1}). The same argument can be used to prove (\ref{e-2}).
\\
Next  we define the function
\[
{\mathfrak c}_{N,M}(t,s)={\mathfrak c}(t,s)
+\frac{N^2L}{M^2R}{\mathfrak c}_*(t,s).
\]
It is clear that ${\mathfrak c}_{N,M}$
 is a covariance function and therefore we
can find $\bar{\kappa}_1=\bar{\kappa}_{1}(N,M)\geq
\bar{\kappa}_2=\bar{\kappa}_{2}(N,M)\geq \ldots $ and orthonormal
functions $\bar{u}_1(t)=\bar{u}_{1}(N,M),
\bar{u}_2(t)=\bar{u}_{2}(N,M),\ldots$ satisfying
\[
\bar{\kappa}_i \bar{u}_i(t)
=\int {\mathfrak c}_{N,M}(t,s)\bar{u}_i(s)ds, \;\;\;1\leq i <\infty.
\]

Now we define the vector
$$
\bpsi_i=(\la A_i, \bar{u}_1\ra/(R\bar{\kappa}_1)^{1/2}, \la A_i, \bar{u}_2\ra/(R\bar{\kappa}_2)^{1/2},\dots ,\la A_i, \bar{u}_d\ra/(R\bar{\kappa}_d)^{1/2})^T, \;\;\;1\leq i \leq K.
$$
It is easy to see that $\bpsi_i,\,1\leq i \leq K$, are independent
and identically distributed random vectors with mean ${\bf 0}$
and $E\bpsi_1\bpsi_1^T={\bf I}_d$,
where ${\bf I}_d$ is the $d\times d$ identity matrix.
Also, (\ref{e-1}) and (\ref{e-2}) imply that
$$
E|\bpsi_1|\leq c_4\left(\sum_{\ell=1}^d1/\bar{\kappa}_\ell\right)^{3/2},
$$
where $c_4$ only depends on $E\norm Z_1\norm^3$ and $E\norm Q_1\norm^3.$ Using Lemma \ref{sena} we obtain similarly to (\ref{eq-2.7}) that there are  independent standard normal random vectors  $\bgamma_i=\bgamma_i(N,M), 1\leq i \leq K,$ in $R^d$ such that
\begin{align}\label{two-sena}
P\biggl\{\biggl|\sum_{i=1}^K\bpsi_i-\sum_{i=1}^K\bgamma_i\biggl|\geq c_5K^{3/8}d^{1/4}\biggl(\sum_{\ell=1}^d1/\bar{\kappa}_\ell &\biggl)^{3/8}\biggl\}\\
&\leq c_5K^{-1/8}d^{1/4}\left(\sum_{\ell=1}^d1/\bar{\kappa}_\ell\right)^{3/8},\notag
\end{align}
where $c_5$ does not depend on $d$. Let
$$
\tilde{\bpsi}=(\la \tilde{A}, \bar{u}_1\ra/\sqrt{\bar{\kappa}_1}, \la \tilde{A}, \bar{u}_2\ra/\sqrt{\bar{\kappa}_2},\ldots , \la\tilde{A}, \bar{u}_d\ra/\sqrt{\bar{\kappa}_d})^T.
$$
It follows from (\ref{e-1}) and (\ref{e-2}) that with some constant $c_6$,
not depending on $d$ we have
$$
E|\tilde{\bpsi}|^3\leq c_6K^{3/2}\left(\sum_{\ell=1}^d1/\bar{\kappa}_\ell\right)^{3/2}
$$
and therefore by Markov's inequality for every $x>0$
\beq\label{two-eq-2}
P\biggl\{N^{-1/2}|\tilde{\bpsi}|>x\biggl\}\leq c_7\frac{K^{3/2}}{x^3 N^{3/2}}\left(\sum_{\ell=1}^d1/\bar{\kappa}_\ell\right)^{3/2}.
\eeq
Let
$$
\bkappa_{N,M}=(\la F_{N,M}, \bar{u}_1\ra /\sqrt{\bar{\kappa}_1}, \la F_{N,M}, \bar{u}_2\ra /\sqrt{\bar{\kappa}_2},\ldots ,\la F_{N,M}, \bar{u}_d\ra /\sqrt{\bar{\kappa}_d})^T.
$$
Next we choose $K=\lfloor N^{3/4} \rfloor$ in (\ref{two-sena}), (\ref{two-eq-2}) and $x=K^{-1/8}(\sum_{\ell=1}^d1/\bar{\kappa}_\ell)^{3/8}$ in (\ref{two-eq-2}) to conclude that there is $\bgamma_{N,M}$, a standard normal random vector in $R^d$ such that
\begin{align}\label{two-eq-3}
P\Biggl\{ \left|\frac{1}{\sqrt{N^*}}\bkappa_{N,M}-\bgamma_{N,M}  \right| &\geq c_8N^{-3/32}d^{1/4}\left(\sum_{\ell=1}^d1/\bar{\kappa}_\ell\right)^{3/8} \Biggl\}\\
&\leq
c_8N^{-3/32}d^{1/4}\left(\sum_{\ell=1}^d1/\bar{\kappa}_\ell\right)^{3/8},\notag
\end{align}
where $N^*=\lfloor N/\lfloor N^{3/4}\rfloor  \rfloor\lfloor N^{3/4}\rfloor$.
Using the definitions of ${\mathfrak c}_P$ and
 ${\mathfrak c}_{N,M}$, together with Assumption \ref{m-1}, we conclude
\beq\label{two-eq-4}
\norm {\mathfrak c}_P - {\mathfrak c}_{N,M}\norm=O(N^{-1/4}),
\eeq
so by Lemma 2.3 of \citet{HKbook}, cf. Lemma \ref{dunsch},  we have
\beq\label{two-eq-5}
|\kappa_i-\bar{\kappa}_i|\leq c_9 \,\norm {\mathfrak c}_P - {\mathfrak c}_{N,M}\norm=O(N^{-1/4}).
\eeq
Using Assumption \ref{m-3} we conclude that
$$
\sum_{\ell=1}^d1/\bar{\kappa_\ell}=O\left(\sum_{\ell=1}^d1/{\kappa_\ell}\right).
$$
Hence it follows from (\ref{two-eq-3}) and Assumption \ref{m-3} that
$$
\frac{1}{N}|\bkappa_{N,M}|^2-\frac{N^*}{N}|\bgamma_{N,M}|^2=o_P(d^{1/2}).
$$
Since $|\bgamma_{N,M}|^2$ is a $\chi^2$ random variable with  $d$
degrees of freedom, Assumption \ref{m-3} yields that
$$
\left|\frac{N^*}{N}-1\right||\bgamma_{N,M}|^2=o_P(d^{1/2}).
$$
It is well known that   $(|\bgamma_{N,M}|^2-d)/(2d)^{1/2}$ converges in distribution to a standard normal random variable, and therefore
\[
\frac{1}{\sqrt{2d}}\left\{\frac{1}{N}|\bkappa_{N,M}|^2-d\right\}\;\stackrel{{\mathcal D}}{\to}\;\;N(0,1),
\]
where $N(0,1)$ stands for a standard normal random variable.\\ The
difference between $|\bkappa_{N,M}|^2/N$ and $\widehat{D}_{N,M}$ is that
the projections are done into the direction of different functions
($\bar{u}_i$'s and $\hat{u}_i$'s, respectively) and the normalizations
($\bar{\kappa}_i$'s and $\hat{\kappa}_i$'s, respectively) are also
different. However, using the Marcinkiewicz--Zygmund law of large numbers in
a Banach space together with \eqref{two-eq-4} and Assumption \ref{m-3}, we obtain that
\[
\norm \hat{\mathfrak c}_P - {\mathfrak c}_{N,M} \norm=O_P(N^{-1/4}).
\]
Hence, in view of \eqref{two-eq-5}, also
$$
\sup_{i}|\hat{\kappa}_i-\bar{\kappa}_i|=O_P(N^{-1/4}),
$$
and there are random signs $\hat{d}_i$ such that
$$
\sup_i\left(\sum_{\ell =1}^i1/\iota_\ell\right)^{-1}\norm \hat{u}_i-\hat{d}_i\bar{u}_i \norm=O_P(N^{-1/4}).
$$
So repeating the arguments used in the proof of Theorem \ref{th-1},
we get
$$
\left|\widehat{D}_{N,M}-\frac{1}{N}|\bkappa_{N,M}|^2\right|=o_P(d^{1/2}),
$$
completing the proof.

\bibliographystyle{plainnat}

\renewcommand{\baselinestretch}{1.0}
\small


\end{document}